\documentclass[preprint,12pt]{elsarticle}

\usepackage{amssymb}
\usepackage{amsmath}
\usepackage{amsthm}

\usepackage{amsfonts}
\usepackage{graphicx}
\usepackage{enumerate}
\usepackage{hyperref}%

\newtheorem{example}{Example}
\newtheorem{definition}{Definition}
\newtheorem{corollary}{Corollary}

\newtheorem{proposition}{Proposition}

\newtheorem{remark}{Remark}

\newtheorem{theorem}{Theorem}
\newtheorem{lemma}{Lemma}

\journal{Journal of Differential Equations}

\begin{document}

\begin{frontmatter}



\title{Strong chain control sets and affine control systems} 

\author{Fritz Colonius\corref{cor1}\fnref{label1}}
\ead{fritz.colonius@uni-a.de}
\cortext[cor1]{Corresponding author}
 \affiliation[label1]{organization={Institut 
 fur Mathematik},
             addressline={Universitatsstra\ss e 12},
             city={Augsburg},
             postcode={86159},
             country={Germany}}

  \author{Alexandre J. Santana\fnref{label2,fn1}}
 \fntext[label2]{Partially supported by CNPq grant n. 309409/2023-3}
 \affiliation[label2]{organization={Department of Mathematics},
             addressline={Av. Colombo, 5790},
              city={Maringa},
              postcode={87020-900},
             country={Brazil}}

\begin{abstract}
 For control-affine systems on non-compact manifolds, the
notion of strong chain control sets is introduced and related to the strong chain
transitivity of the associated control flows. Affine control systems on
$\mathbb{R}^{n}$ are embedded into bilinear control systems in an extended
state space and it is shown that they are topologically conjugate to the
induced system on the northern hemisphere of the Poincar\'{e} sphere. This
preserves strong chain control sets. Further chain controllability properties
on spheres are analyzed.
\end{abstract}



\begin{keyword}
affine control system \sep strong chain control set \sep strong
chain transitivity \sep \ Poincar\'{e} sphere


93B05 \sep 37B20

\end{keyword}

\end{frontmatter}



\section{Introduction}

This paper is mainly motivated by the problem of describing generalized
controllability properties for affine control systems on $\mathbb{R}^{n}$ of
the form%
\begin{equation}
\dot{x}(t)=A_{0}x(t)+a_{0}+\sum_{i=1}^{m}u_{i}(t)[A_{i}x(t)+a_{i}%
],\,u\in\mathcal{U}, \label{affine1}%
\end{equation}
where $A_{i}\in\mathbb{R}^{n\times n},a_{i}\in\mathbb{R}^{n},i\in
\{0,1,\ldots,m\}$. The control range $\Omega\subset\mathbb{R}^{m}$ is a convex
and compact neighborhood of $0\in\mathbb{R}^{m}$, and the controls $u$ are in%
\begin{equation}
\mathcal{U}:=\left\{  u\in L^{\infty}(\mathbb{R},\mathbb{R}^{m})\left\vert
u(t)\in\Omega\text{ for almost all }t\in\mathbb{R}\right.  \right\}  .
\label{U}%
\end{equation}
We denote the solution for initial condition $x(0)=x_{0}\in\mathbb{R}^{n}$ and
control $u\in\mathcal{U}$ by $\psi(t,x_{0},u),t\in\mathbb{R}$. The homogeneous
part of this system is the bilinear control system%
\begin{equation}
\dot{x}(t)=[A_{0}+\sum_{i=1}^{m}u_{i}(t)A_{i}]x(t),\,u\in\mathcal{U}.
\label{bilinear0}%
\end{equation}
Controllability questions for bilinear and affine control systems are a
classical topic of control theory. Early contributions are due to Rink and
Mohler \cite{RinkM68} who took the set of equilibria as a starting point for
establishing results on complete controllability. A theorem due to Jurdjevic
and Sallet \cite[Theorem 2]{JurS84} (cf. also Do Rocio, Santana, and Verdi
\cite{DSC09}) shows that affine system (\ref{affine1}) is controllable on
$\mathbb{R}^{n}$ if it has no fixed points and its homogeneous part
(\ref{bilinear0}) is controllable on $\mathbb{R}^{n}\setminus\{0\}$. Cannarsa
and Sigalotti \cite{CanS21} prove that approximately controllable bilinear
systems are controllable. As Jurdjevic \cite[p. 182]{Jurd97} emphasizes, the
controllability properties of affine systems are substantially richer than for
linear systems and may require \textquotedblleft entirely different
geometrical considerations\textquotedblright, based on Lie-algebraic methods,
which have yielded important insights. There is substantial literature on
controllability properties of bilinear and affine control systems. Here we
only refer to the monographs Mohler \cite{Mohler}, Elliott \cite{Elliott}, and
Jurdjevic \cite{Jurd97}.

In the present paper, we study chain controllability properties, which
constitute a weaker version of approximate controllability in infinite time.
Here small jumps in the trajectories are allowed. The approach we propose is
to extend the affine system to a bilinear control system on $\mathbb{R}^{n+1}%
$\ (cf. Jurdjevic \cite{Jurd84} for this idea in the context of the
Lie-algebraic approach). This system can be projected to the unit sphere
$\mathbb{S}^{n}$ or to projective space $\mathbb{P}^{n}$. We call
$\mathbb{S}^{n}$ the Poincar\'{e} sphere since a similar construction in the
theory of ordinary differential equation goes back to Poincar\'{e} \cite{Poin}
(cf. Perko \cite{Perko}). The induced system on the open northern hemisphere
$\mathbb{S}^{n,+}$ is conjugate to the system on $\mathbb{R}^{n}$, hence one
can study the chain controllability properties on the sphere. This continues
our previous work, where we consider control sets for affine system (Colonius,
Santana, Setti \cite{ColSS23}), chain control sets (i.e., maximal regions of
chain controllability) for linear control systems (Colonius, Santana, and
Viscovini \cite{ColSV24}) and for affine control systems (Colonius and Santana
\cite{ColS24a}), where we show that the chain control sets in $\mathbb{R}^{n}$
are mapped to subsets of the central chain control set of the induced system
on projective space $\mathbb{P}^{n}$.

The proposed approach for affine control systems is confronted with several
obstacles, which we can only partially surmount.

In general, chain controllability is only preserved under uniformly continuous
conjugation maps, which does not hold for the conjugation between the
non-compact spaces $\mathbb{R}^{n}$ and $\mathbb{S}^{n,+}$. We introduce a
modified version of chain controllability, which we term strong chain
controllability. Instead of bounding the jump lengths of controlled chains by
constants $\varepsilon>0$ we consider continuous jump functions defined on the
state space with values in $(0,\infty)$. This is based on an approach by
Hurley to dynamical systems on non-compact state spaces; cf. \cite{Hur92,
Hur95} and also Choi, Chu, and Park \cite{ChoiCP02}. Hurley's approach is
motivated by the desire to develop a theory, that only depends on the
topology and not on the metric. We apply it to the associated control flows.
Note that a different strengthening of chain transitivity (also called strong
chain transitivity) is due to Easton \cite{Easton}, who proposes to sum up the
length of the jumps along an $\varepsilon$-chain. This has e.g. been explored
by Ahmadi \cite{Ahma24}, where also further references are given. Patr\~{a}o
and San Martin \cite{PatSM07} treat chain transitivity in topological spaces
based on families of coverings by constructing associated shadowing semigroups.

Strong chain control sets are contained in chain control sets and coincide
with them when the state space is compact. For non-compact state spaces, the
use of jump functions instead of constant jump bounds entails that close to
the boundary of the state space or close to \textquotedblleft
infinity\textquotedblright\ very small jump lengths are required. Definition
\ref{Definition_strong} introduces the notion of strong chain control sets,
which allows us to show that the strong chain control sets of affine system
(\ref{affine1}) on $\mathbb{R}^{n}$ are mapped onto the strong chain control
sets of the induced system on the open northern hemisphere $\mathbb{S}^{n,+}$
of the Poincar\'{e} sphere (Corollary \ref{Corollary_relation}). The same is
true for the southern hemisphere $\mathbb{S}^{n,-}$. Hence the problem of
determining the strong chain control sets of affine system (\ref{affine1}) is
transformed into the problem of describing the strong chain control sets on
the hemispheres. For the latter problem, we have only very partial results. We
can describe the chain control sets for the systems on the sphere induced by
general bilinear control systems. Furthermore, also for the system restricted
to the closed northern hemisphere $\overline{\mathbb{S}^{n,+}}$ some results
are derived.

Section \ref{Section2} cites a fundamental lemma due to Hurley \cite{Hur95}
and uses it to describe the behavior of strongly chain transitive sets of
flows under topological conjugation. For general control-affine systems, the
relation of chain control sets to maximal chain transitive sets of the
associated control flows is recalled. Affine systems of the form
(\ref{affine1}) are embedded into bilinear control systems on an extended
state space. The associated control flow possesses a Selgrade decomposition,
which is closely related to the Selgrade decomposition of the homogeneous part
(\ref{bilinear0}). This yields that there is a unique central chain control
set $_{\mathbb{P}}E_{c}$ on the projective Poincar\'{e} sphere that is not
contained in the projective equator. These results are taken from Colonius and
Santana \cite{ColS24a}. Section \ref{Section3} introduces the notion of strong
chain controllability and strong chain control sets for control-affine systems
on non-compact manifolds $M$. Theorem \ref{Theorem_equivalence} shows that the
lifts to $\mathcal{U}\times M$ of the strong chain control sets are the
maximal invariant strongly chain transitive sets of the control flow. Example
\ref{Example_linear} presents a linear control system, where the control set
around the origin is inside the strong chain control set $E^{\ast}$ and
$E^{\ast}$ is a proper subset of the chain control set $E$. Section
\ref{Section4} establishes that every affine control system on $\mathbb{R}%
^{n}$ is conjugated to the induced system on the northern hemisphere
$\mathbb{S}^{n,+}$ of the Poincar\'{e} sphere and that the strong chain
control sets are preserved under this conjugacy. Analogous results hold for
the conjugation to the induced system on a subset of $\mathbb{P}^{n}$; cf.
Theorem \ref{Theorem_projective}. The ensuing sections contribute to the
analysis of chain controllability on spheres. For general bilinear control
systems, Section \ref{Section5} describes the chain control sets for the
induced systems on the sphere. Theorem \ref{Theorem_twoS} proves that for any
chain control set in projective space, the preimage on the sphere is either a
single chain control set or the union of two chain control sets. This applies,
in particular, to the central chain control set $_{\mathbb{P}}E_{c}%
\subset\mathbb{P}^{n}$ of an affine control system. Section \ref{Section6}
analyzes the chain control sets on the closed hemispheres of the Poincar\'{e}
sphere. Here additional assumptions are imposed on the chain control sets
$_{\mathbb{S}}E^{\hom}$ of the system on $\mathbb{S}^{n-1}$ induced by the
homogeneous part (\ref{bilinear0}). They determine the chain control sets of
the system restricted to the equator $\mathbb{S}^{n,0}$ of the Poincar\'{e}
sphere $\mathbb{S}^{n}$.

\textbf{Notation.} The origin in $\mathbb{R}^{n}$ is denoted by $0_{n}$ and
$0_{1}$ is abbreviated by $0$. The projections from $\mathbb{R}^{n}$ to the
sphere $\mathbb{S}^{n-1}$ and projective space $\mathbb{P}^{n-1}$ are denoted
by $\mathbb{S}x$ and $\mathbb{P}x$ for $0_{n}\not =x\in\mathbb{R}^{n}$.

\section{Preliminaries\label{Section2}}

In the first subsection, we recall properties of strong chain recurrence for
flows on non-compact metric spaces. The second subsection collects notions and
results for control systems and the associated control flows. In particular,
some properties of affine control systems are reviewed.

\subsection{Strong chain recurrence for flows\label{Subsection2.1}}

We formulate a lemma due to Hurley \cite{Hur95} characterizing strong chain
recurrence for continuous flows $\phi$ on non-compact metric spaces $X$ with
metric $d$, hence $\phi:\mathbb{R}\times X\rightarrow X$ is continuous with
$\phi(0,x)=x$ and $\phi(t+s,x)=\phi(t,\phi(s,x))$ for $t,s\in\mathbb{R}$ and
$x\in X$. Furthermore, an example illustrates the difference between chain
recurrence and strong chain recurrence.

Let $\mathcal{P}(X):=\{\varepsilon:X\rightarrow(0,\infty),~$continuous$\}$.
Given a function $\varepsilon\in\mathcal{P}(X)$ and $T>0$ sequences
$x_{0}=x,x_{1},\ldots,x_{k-1},x_{k}=y$ in $X$ and $t_{0},\ldots,t_{k-1}\geq T$
form an $(\varepsilon,T)$-chain from $x$ to $y$ if $d(\phi(t_{j}%
,x_{j}),x_{j+1})<\varepsilon(\phi(t_{j},x_{j}))$ for\thinspace$\,j=0,\ldots
,k-1$. A point $x\in X$ is called strongly chain recurrent if for all $T>0$
and $\varepsilon\in\mathcal{P}(X)$ there is an $(\varepsilon,T)$-chain from
$x$ to $x$. A set $Y\subset X$ is called strongly chain transitive, if for all
$x,y\in Y$ and all $T>0$ and $\varepsilon\in\mathcal{P}(X)$ there is an
$(\varepsilon,T)$-chain from $x$ to $y$. The set of all strongly chain
recurrent points is denoted by $\mathcal{CR}^{\ast}(X)$. If $X$ is compact,
the set $\mathcal{CR}^{\ast}(X)$ coincides with the chain recurrent set
$\mathcal{CR}(X)$, i.e., the set of all points $x$ such that for all constants
$\varepsilon>0$ and all $T>0$ there is an $(\varepsilon,T)$-chain from $x$ to
$x$ (this can be seen similarly as Proposition \ref{Proposition_compact}).

In the following the symbol \textquotedblleft$\varepsilon$\textquotedblright%
\ often denotes a function. It will be clear from the context when a constant
$\varepsilon>0$ is meant. We will need the following fundamental lemma due to
Hurley \cite[Lemma 2]{Hur95}; cf. also Choi, Chu, and Park \cite[Lemma
2]{ChoiCP02}.

\begin{lemma}
\label{Lemma2_Hurley}Suppose that $(X,d)$ and $(Y,\rho)$ are metric spaces and
that $g:X\rightarrow Y$ and $\varepsilon:Y\rightarrow(0,\infty)$ are
continuous. Then there is a continuous map $\delta:X\rightarrow(0,\infty)$
such that $d(x,y)<\delta(x)$ implies $\rho(g(x),g(y))<\varepsilon(g(x))$.
\end{lemma}

This lemma has the following consequence for strong chain transitivity.

\begin{theorem}
\label{Theorem_components}Let $\phi$ and $\psi$ be continuous flows on metric
spaces $(X,d)$ and $(Y,\rho)$, respectively. Suppose that $h:X\rightarrow Y$
is a homeomorphism conjugating $\phi$ and $\psi$, i.e.,
\[
h(\phi(t,x))=\psi(t,h(x))\text{ for all }x\in X\text{ and }t\in\mathbb{R}%
\text{.}%
\]
Then $\mathcal{M}\subset X$ is a maximal invariant strongly chain transitive
set for $\phi$ if and only if $h(\mathcal{M})$ is a maximal invariant strongly
chain transitive set for $\psi$.
\end{theorem}

\begin{proof}
Let $\varepsilon\in\mathcal{P}(Y)$ and choose $\delta\in\mathcal{P}(X)$
according to Hurley's lemma, Lemma \ref{Lemma2_Hurley}. Consider
$y=h(x),y^{\prime}=h(x^{\prime})\in h(\mathcal{M})$ and let $T>0$. Since
$\mathcal{M}$ is strongly chain transitive there is a $(\delta,T)$-chain in
$X$ given by $x_{0}=x,x_{1},\ldots,x_{k}=x^{\prime}$ in $X$ and $T_{0}%
,\ldots,T_{k-1}\geq T$ with%
\[
d(\phi(T_{j},x_{j}),x_{j+1})<\delta(\phi(T_{j},x_{j}))\text{ for }%
\,j=0,\ldots,k-1.
\]
Then the points $y_{j}:=h(x_{j}),\,j=0,\ldots,k-1$, satisfy%
\begin{align*}
\rho(\psi(T_{j},y_{j}),y_{j+1})  &  =\rho(\psi(T_{j},h(x_{j})),h(x_{j+1}%
))=\rho(h(\phi(T_{j},x_{j})),h(x_{j+1}))\\
&  <\varepsilon(h(\phi(T_{j},x_{j}))=\varepsilon(\psi(T_{j},h(x_{j}%
)))=\varepsilon(\psi(T_{j},y_{j})).
\end{align*}
It follows that an $(\varepsilon,T)$-chain in $Y$ is given by $y_{0}%
=y=h(x),y_{1}=h(x_{1}),\ldots,\allowbreak y_{k}=y^{\prime}=h(x^{\prime})$ and
$T_{0},\ldots,T_{k-1}\geq T$. Since $\mathcal{M}$ is an invariant strongly
chain transitive set for $\phi$ it follows that $h(\mathcal{M})$ is an
invariant strongly chain transitive for $\psi$. The converse follows by
considering $h^{-1}$. Thus also the maximality property follows.
\end{proof}

In the following example, the state space is chain transitive but not strongly
chain transitive.

\begin{example}
\label{Example_strong}Consider the following two-dimensional autonomous
differential equation%
\[
\left(
\begin{array}
[c]{c}%
\dot{x}(t)\\
\dot{y}(t)
\end{array}
\right)  =\left(
\begin{array}
[c]{cc}%
0 & 1\\
0 & 0
\end{array}
\right)  \left(
\begin{array}
[c]{c}%
x(t)\\
y(t)
\end{array}
\right)  .
\]
In the following, we show that the strong chain recurrent set equals the
$x$-axis $\mathbb{R}\times\{0\}$.

It is clear that the $x$-axis is contained in the strong chain recurrent set
since it consists of equilibria. The components of the solutions are $\phi
_{1}(t,x,y)=x+ty,\phi_{2}(t,x,y)=y$ for $t\in\mathbb{R}$. Hence in the open
upper half-plane, the trajectories move to the right on parallels to the
$x$-axis. The state space $\mathbb{R}^{2}$ is easily seen to be chain
transitive: For example, for constant $\varepsilon>0$ and $T>0$ any
$(\varepsilon,T)$-chain from an initial point $(0,y_{0}),y_{0}>0$, on the
$y$-axis to the $x$-axis follows parallels to the $x$-axis till time $T$,
jumps to a lower parallel to the $x$-axis, etc.

\textbf{Claim:} There are $T>0$ and $\varepsilon\in\mathcal{P}(\mathbb{R}%
^{2})$ such that there is no $(\varepsilon,T)$-chain from $(x_{0}%
,y_{0})=(0,2)$ to the line $y=1$.

Let $T>0,\varepsilon\in\mathcal{P}(\mathbb{R}^{2})$ with $\varepsilon
(x,y)=\varepsilon^{\prime}(x)+\delta(x,y),$ where $\varepsilon^{\prime
}:\mathbb{R}\rightarrow(0,\infty)$ and $\delta:\mathbb{R}^{2}\rightarrow
(0,\infty)$ are continuous with $\delta(0,y)>0$ for all $y\in\mathbb{R}$ and
$\delta(x,y)=0$ for $\left\vert x\right\vert \geq1,y\in\mathbb{R}$.
Furthermore, we assume that $\varepsilon^{\prime}(x)$ strongly decreases to
$0$ for $x\rightarrow\infty$.

Let the jump points be $(x_{i},y_{i})$ with $y_{i}\in(1,2]$ and $x_{i+1}%
=\phi_{1}(T_{i},x_{i},y_{i})$. It suffices to show that $Y:=\lim
_{i\rightarrow\infty}y_{i}>1$, hence the line $y=1$ is not reached.

Note that $\varphi_{1}(T_{0},0,2)=2T_{0}>1$ for $T>\frac{1}{2}$, and%
\begin{align*}
\varepsilon(\phi(T_{0},0,2))  &  =\varepsilon(\phi_{1}(T_{0},0,2),\phi
_{2}(T_{0},0,2))=\varepsilon^{\prime}(\phi_{1}(T_{0},0,2))+\delta(\phi
_{1}(T_{0},0,2),2)\\
&  =\varepsilon^{\prime}(2T_{0})\leq\varepsilon^{\prime}(2T).
\end{align*}
Hence in the following only the function $\varepsilon^{\prime}$ will be
estimated. The first jump to $(x_{1},y_{1})$ occurs at $x_{1}=\phi_{1}%
(T_{0},0,2)=2T_{0}$ and%
\[
\phi_{2}(T_{0},y_{0})-y_{1}=y_{0}-y_{1}<\varepsilon^{\prime}(2T_{0})\text{,
i.e., }y_{1}>y_{0}-\varepsilon^{\prime}(2T_{0})\geq2-\varepsilon^{\prime
}(2T).
\]
We may assume that $y_{1}>2-\varepsilon^{\prime}(2T)>1$. In general a jump
occurs at%
\[
\phi_{2}(T_{i},x_{i},y_{i})-y_{i+1}=y_{i}-y_{i+1}<\varepsilon^{\prime}%
(\phi_{1}(T_{i},x_{i},y_{i}))=\varepsilon^{\prime}(x_{i}+T_{i}y_{i}),
\]
hence $y_{i+1}>y_{i}-\varepsilon^{\prime}(x_{i}+T_{i}y_{i})$ and
$x_{i+1}=x_{i}+T_{i}y_{i}$. Since $\varepsilon^{\prime}(x_{i}+T_{i}%
y_{i})<\varepsilon^{\prime}(x_{i})$ it follows that $y_{i+1}>y_{i}%
-\varepsilon^{\prime}(x_{i})$ implying%
\[
y_{i}>y_{1}-\sum\nolimits_{j=1}^{i-1}\varepsilon^{\prime}(x_{j})\text{ for
}i\geq2.
\]
The $x$-component satisfies $x_{i+1}=x_{i}+T_{i}y_{i}\geq x_{i}+T$ and hence
$x_{i}>x_{1}+(i-1)T$ for $i\geq2$. Since $y_{1}-1>0$ we can, additionally,
suppose that for the function $\varepsilon^{\prime}$ the following limit
exists and satisfies%
\[
\alpha:=\sum\nolimits_{j=1}^{\infty}\varepsilon^{\prime}(x_{1}+(j-1)T)<y_{1}%
-1.
\]
This implies $\sum_{j=1}^{i-1}\varepsilon^{\prime}(x_{j})<\sum_{j=1}%
^{i-1}\varepsilon^{\prime}(x_{1}+(j-1)T)\leq\alpha$. The claim follows from%
\[
y_{i}>y_{1}-\sum\nolimits_{j=1}^{i-1}\varepsilon^{\prime}(x_{j})>y_{1}%
-\alpha>1\text{ for all }i\geq2.
\]
A consequence of the claim is that there is no $(\varepsilon,T)$-chain from
$(x_{0},y_{0})=(0,2)$ to the $x$-axis.

Analogous, but technically more involved arguments show for every point
$(x_{0},y_{0})$ in the open upper and the open lower half-planes that there
are $T>0$ and $\varepsilon\in\mathcal{P}(\mathbb{R}^{2})$ such that there is
no $(\varepsilon,T)$-chain from $(x_{0},y_{0})$ to the $x$-axis. One concludes
that the strong chain recurrent set coincides with the $x$-axis.
\end{example}

\subsection{Control systems\label{Subsection2.2}}

Consider control-affine systems of the form
\begin{equation}
\dot{x}(t)=X_{0}(x(t))+\sum_{i=1}^{m}u_{i}(t)X_{i}(x(t)),\,u\in\mathcal{U},
\label{control_affine}%
\end{equation}
where $X_{0},X_{1},\ldots,X_{m}$ are smooth ($C^{\infty}$-)vector fields on a
smooth manifold $M$ and $\mathcal{U}$ is defined by (\ref{U}). We assume that
for every control $u\in\mathcal{U}$ and every initial state $x(0)=x_{0}\in M$
there exists a unique (Carath\'{e}odory) solution $\varphi(t,x_{0}%
,u),t\in\mathbb{R}$.

Recall the notion of chain controllability (cf. Colonius and Kliemann \cite[p.
80]{ColK00}, Kawan \cite[Definition 1.14]{Kawa13}), which allows for (small)
jumps between pieces of trajectories; cf. also Kawan \cite{Kawa16} and Ayala,
Da Silva, and San Martin \cite{AyDSSM17}. Here a metric $d_{M}$ compatible
with the topology on $M$ is fixed.

\begin{definition}
\label{intro2:defchains}Let $x,y\in M$. For $\varepsilon,T>0$ a controlled
$(\varepsilon,T)$\textit{-chain} $\zeta$ from $x$ to $y$ is given by
$k\in\mathbb{N},\ x_{0}=x,x_{1},\ldots,x_{k}=y\in M,\ u_{0},\ldots,u_{k-1}%
\in\mathcal{U}$, and $T_{0},\ldots,T_{k-1}\geq T$ with
\[
d_{M}(\varphi(T_{j},x_{j},u_{j}),x_{j+1})<\varepsilon\text{ }%
\,\text{for\thinspace all}\,\,j=0,\ldots,k-1.
\]
If for every $\varepsilon,T>0$ there is a controlled $(\varepsilon,T)$-chain
from $x$ to $y$, the point $x$ is chain controllable to $y$. The chain
reachable set from $x\ $and the chain controllable set to $x$ are
\begin{align*}
\mathbf{R}^{c}(x)  &  =\left\{  y\in M\left\vert x\text{ is chain controllable
to }y\right.  \right\}  ,\\
\mathbf{C}^{c}(x)  &  =\left\{  y\in M\left\vert y\text{ is chain controllable
to }x\right.  \right\}  ,
\end{align*}
respectively. A nonvoid set $F\subset M$ is said to be chain controllable, if
$x$ is chain controllable to $y$ for all $x,y\in F$.
\end{definition}

Chain controllability is a generalized version of approximate controllability
in infinite time. We define chain control sets as maximal chain controllable sets.

\begin{definition}
\label{Definition_chain_control}A nonvoid set $E\subset M$ is a \textit{chain
control set} of system (\ref{control_affine}) if for all $x,y\in E$ and
$\varepsilon,T>0$ there is a controlled $(\varepsilon,T)$-chain from $x$ to
$y$, and $E$ is maximal with this property.
\end{definition}

We cite the following result from Colonius, Santana, and Viscovini
\cite[Theorem 2.14 (iii)]{ColSV24}.

\begin{theorem}
\label{Theorem_E_omega}Let $E$ be a chain control set. Then it follows for all
$x\in E$ that $E=\mathbf{R}^{c}(x)\cap\mathbf{C}^{c}(x)$ and that there exists
$u\in\mathcal{U}$ such that $\varphi(t,x,u)\in E$ for all $t\in\mathbb{R}$.
\end{theorem}

Topological conjugacies between control systems are defined as follows. For
simplicity, we suppose that the sets of control functions coincide.

\begin{definition}
\label{Definition_conjugacy}Consider two control systems of the form
(\ref{control_affine}) on manifolds $M$ and $N$, respectively, with controls
$u\in\mathcal{U}$, given by%
\begin{equation}
\dot{x}(t)=X_{0}(x(t))+\sum_{i=1}^{m}u_{i}(t)X_{i}(x(t))\text{ and }\dot
{y}(t)=Y_{0}(y(t))+\sum_{i=1}^{m}u_{i}(t)Y_{i}(y(t)). \label{control_con}%
\end{equation}
Denote the trajectories by $\varphi(t,x_{0},u)$ and $\psi(t,y_{0}%
,u),t\in\mathbb{R}$, respectively. A homeomorphism $H:M\rightarrow N$ is a
topological conjugacy of these control systems, if%
\[
H(\varphi(t,x_{0},u))=\psi(t,H(x_{0}),u)\text{ for }t\in\mathbb{R},x_{0}\in
M,u\in\mathcal{U}.
\]

\end{definition}

Control systems of the form (\ref{control_affine}) come with an associated
flow called the control flow. Endow the set $\mathcal{U}$ of controls with a
metric compatible with the weak$^{\ast}$ topology on $L^{\infty}%
(\mathbb{R},\mathbb{R}^{m})$. One may choose%
\begin{equation}
d_{\mathcal{U}}(u,v)=\sum_{i=1}^{\infty}\frac{1}{2^{i}}\frac{\left\vert
\int_{\mathbb{R}}\left\langle u(t)-v(t),y_{i}(t)\right\rangle dt\right\vert
}{1+\left\vert \int_{\mathbb{R}}\left\langle u(t)-v(t),y_{i}(t)\right\rangle
dt\right\vert }, \label{metric_U}%
\end{equation}
where $\{y_{i}\}$ is an arbitrary countable dense set in $L^{1}(\mathbb{R}%
,\mathbb{R}^{m})$. A metric on $\mathcal{U}\times M$ is defined by
$d((u,x),(v,y))=\max\{d_{\mathcal{U}}(u,v),d_{M}(x,y)\}$. The control flow is
defined as%
\begin{equation}
\Phi:\mathbb{R}\times\mathcal{U}\times M\rightarrow\mathcal{U}\times
M,\,(t,u,x_{0})\mapsto(u(t+\cdot),\varphi(t,x_{0},u)), \label{cflow}%
\end{equation}
where $u(t+\cdot)(s):=u(t+s),s\in\mathbb{R}$, is the right shift. The control
flow $\Phi$ is continuous and $\mathcal{U}$ is compact and chain transitive;
cf. Colonius and Kliemann \cite[Chapter 4]{ColK00} or Kawan \cite[Section
1.4]{Kawa13}.

Chain control sets can be characterized as maximal chain transitive sets of
the control flow. By \cite[Theorem 4.3.11]{ColK00} or \cite[Proposition
1.24(iv)]{Kawa13} a chain control set $E$ yields a maximal invariant chain
transitive set $\mathcal{E}$ of the control flow $\Phi$ via
\begin{equation}
\mathcal{E}:=\{(u,x)\in\mathcal{U}\times M\left\vert \varphi(t,x,u)\in
E\mbox{
			for all }t\in\mathbb{R}\right.  \}, \label{chain_transitive1}%
\end{equation}
and for any maximal invariant chain transitive set in $\mathcal{U}\times M$
the projection to $M$ is a chain control set.

The affine control system (\ref{affine1}) and its homogeneous, linear part
(\ref{bilinear0}) generate control flows which we denote by $\Psi$ and
$\Phi^{\hom}$, respectively, on the vector bundle $\mathcal{V}:=\mathcal{U}%
\times\mathbb{R}^{n}$. Here $\Phi^{\hom}$ is a linear flow. Fundamental
information on linear flows on vector bundles is provided by Selgrade's
theorem; cf. Selgrade \cite{Selg75}, Salamon and Zehnder \cite{SalZ88},
Colonius and Kliemann \cite{ColK00, ColK14}. It provides the finest
decomposition of the underlying vector bundle into exponentially separated
subbundles and, equivalently, the finest Morse decomposition of the associated
flow on the projective bundle $\mathcal{U}\times\mathbb{P}^{n-1}$, which is
compact. The Morse sets in this decomposition are the maximal invariant chain
transitive subsets.

For the linear flow $\Phi^{\hom}$, Selgrade's theorem yields the decomposition%
\begin{equation}
\mathcal{U}\times\mathbb{R}^{n}=\mathcal{V}_{1}\oplus\cdots\oplus
\mathcal{V}_{\ell}\text{ with }1\leq\ell\leq n, \label{Sel_hom}%
\end{equation}
where the Selgrade bundles $\mathcal{V}_{i}$ are exponentially separated and
the projections $\mathbb{P}\mathcal{V}_{i}$ to the projective bundle
$\mathbb{P}\mathcal{V}$ are the maximal invariant chain transitive sets of the
induced flow $\mathbb{P}\Phi^{\hom}$ on $\mathbb{P}\mathcal{V}=\mathcal{U}%
\times\mathbb{P}^{n-1}$. The bilinear control system (\ref{bilinear0}) induces
control systems on the unit sphere $\mathbb{S}^{n-1}$ and on projective space
$\mathbb{P}^{n-1}$ (cf., e.g., Colonius and Kliemann \cite[Chapter 6]{ColK00})
with trajectories $\mathbb{P}\varphi(t,\mathbb{P}x_{0},u),t\in\mathbb{R}$, for
$x_{0}\not =0$. The projective flow $\mathbb{P}\Phi^{\hom}$ is the control
flow for the induced control system on $\mathbb{P}^{n-1}$.

Affine control system (\ref{affine1}) can be embedded into a bilinear control
system on $\mathbb{R}^{n+1}$ of the form (cf. Jurdjevic \cite{Jurd84} or
Elliott \cite[Subsection 3.8.1]{Elliott})%

\begin{equation}
\left(
\begin{array}
[c]{c}%
\dot{x}(t)\\
\dot{z}(t)
\end{array}
\right)  =\left(
\begin{array}
[c]{cc}%
A_{0} & a_{0}\\
0 & 0
\end{array}
\right)  \left(
\begin{array}
[c]{c}%
x(t)\\
z(t)
\end{array}
\right)  +\sum_{i=1}^{m}u_{i}(t)\left(
\begin{array}
[c]{cc}%
A_{i} & a_{i}\\
0 & 0
\end{array}
\right)  \left(
\begin{array}
[c]{c}%
x(t)\\
z(t)
\end{array}
\right)  ,u\in\mathcal{U}, \label{hom_d+1}%
\end{equation}
with trajectories denoted by $\psi^{1}(t,x_{0},z_{0},u),t\in\mathbb{R}$. This
control system induces control systems on the unit sphere $\mathbb{S}^{n}$ and
on projective space $\mathbb{P}^{n}$ with trajectories $\mathbb{P}\psi
^{1}(t,\mathbb{P}(x_{0},z_{0}),u),t\in\mathbb{R}$, for $(x_{0},z_{0}%
)\not =(0,0)$. The associated control flow $\Psi^{1}$ on $\mathcal{U}%
\times\mathbb{R}^{n+1}$ is linear.

The following result (cf. Colonius and Santana \cite[Corollary 32]{ColS24a})
shows that the Selgrade decomposition of $\Psi^{1}$ is closely related to the
Selgrade decomposition (\ref{Sel_hom}) of the homogeneous part\textbf{.} The
Selgrade bundles of $\Psi^{1}$ are essentially the Selgrade bundles of the
homogeneous part $\Phi^{\hom}$, with one exception.

\begin{theorem}
\label{Theorem_Cor32}Consider an affine control system of the form
(\ref{affine1}) and the associated affine control flow $\Psi$ on
$\mathcal{U}\times\mathbb{R}^{n}$. For $i\in\{1,\ldots,\ell\}$ let
$\mathcal{V}_{i}\subset\mathcal{U}\times\mathbb{R}^{n}$ be the Selgrade
bundles of the linear flow $\Phi^{\hom}$ associated with control system
(\ref{bilinear0}), and let $\mathcal{V}_{i}^{\infty}=\mathcal{V}_{i}%
\times\{0\}$.

(i) The Selgrade decomposition of the lifted flow $\Psi^{1}$ has the form
\begin{equation}
\mathcal{U}\times\mathbb{R}^{n+1}=\mathcal{V}_{1}^{\infty}\oplus\cdots
\oplus\mathcal{V}_{\ell^{+}}^{\infty}\oplus\mathcal{V}_{c}^{1}\oplus
\mathcal{V}_{\ell^{+}+\ell^{0}+1}^{\infty}\oplus\cdots\oplus\mathcal{V}_{\ell
}^{\infty}, \label{Sel4}%
\end{equation}
for some numbers $\ell^{+},\ell^{0}\geq0$ with $\ell^{+}+\ell^{0}\leq\ell$.

(ii) The subbundle $\mathcal{V}_{c}^{1}$ called the central Selgrade bundle
satisfies
\[
\mathcal{V}_{c}^{1}\cap\left(  \mathcal{U}\times\mathbb{R}^{n}\times
\{0\}\right)  =\bigoplus_{i=\ell^{+}+1}^{i=\ell^{+}+\ell^{0}}\mathcal{V}%
_{i}^{\infty}:=\mathcal{V}_{c}^{\infty}.
\]

(iii) The dimension of $\mathcal{V}_{c}^{1}$ is given by $\dim\mathcal{V}%
_{c}^{1}=1+\dim\mathcal{V}_{c}^{\infty}$, and $\dim\mathcal{V}_{c}^{1}=1$
holds if and only if $\mathcal{V}_{c}^{1}\cap\left(  \mathcal{U}%
\times\mathbb{R}^{n}\times\{0\}\right)  =\mathcal{U}\times\{0_{n}%
\}\times\{0\}$.

(iv) Suppose that (\ref{bilinear0}) is uniformly hyperbolic, i.e.,
$\mathcal{U}\times\mathbb{R}^{n}$ can be split into a stable and an unstable
subbundle. Then the central Selgrade bundle is the line bundle
\begin{equation}
\mathcal{V}_{c}^{1}=\{(u,-re(u,0),r)\in\mathcal{U}\times\mathbb{R}^{n}%
\times\mathbb{R}\left\vert u\in\mathcal{U},r\in\mathbb{R}\right.  \},
\label{6.7}%
\end{equation}
where $e(u,t),t\in\mathbb{R}$, is the unique bounded solution of
(\ref{affine1}) for $u\in\mathcal{U}$, and $\mathbb{P}\mathcal{V}_{c}^{1}$ is
a compact subset of $\mathcal{U}\times\left\{  \mathbb{P}(x,1)\left\vert
x\in\mathbb{R}^{n}\right.  \right\}  $.
\end{theorem}

This theorem also yields information on chain control sets on $\mathbb{P}^{n}%
$, based on the relation between maximal chain transitive sets of control
flows and chain control sets; cf. (\ref{chain_transitive1}). Projective space
$\mathbb{P}^{n}$ can be written as the disjoint union $\mathbb{P}%
^{n}=\mathbb{P}^{n,1}\cup\,\mathbb{P}^{n,0}$, where $\mathbb{P}^{n,1}%
:=\left\{  \mathbb{P}(x,1)\left\vert x\in\mathbb{R}^{n}\right.  \right\}  $
and $\mathbb{P}^{n,0}:=\left\{  \mathbb{P}(x,0)\left\vert 0\not =%
x\in\mathbb{R}^{n}\right.  \right\}  $. Note that $\mathbb{P}^{n,1}$ can be
identified with the northern hemisphere $\mathbb{S}^{n,+}$ of the unit sphere
$\mathbb{S}^{n}$ and $\mathbb{P}^{n,0}$ corresponds to the equator of
$\mathbb{S}^{n}$. A special feature of the system induced on $\mathbb{P}^{n}$
by the bilinear system (\ref{hom_d+1}) is that the sets $\mathbb{P}^{n,1}$ and
$\mathbb{P}^{n,0}$ are invariant. In this context, we have established in
\cite[Theorem 35]{ColS24a} the following relation between the chain control
sets in these spaces.

\begin{theorem}
\label{Theorem_E}Consider an affine control system of the form (\ref{affine1}).

(i) There is a unique chain control set $_{\mathbb{P}}E_{c}$, called the
central chain control set, of the induced control system on the projective
Poincar\'{e} space $\mathbb{P}^{n}$ such that $_{\mathbb{P}}E_{c}%
\cap\mathbb{P}^{n,1}\not =\varnothing$. It is given by%
\[
_{\mathbb{P}}E_{c}=\{\mathbb{P}(x,r)\in\mathbb{P}^{n}\left\vert \exists
u\in\mathcal{U}:(u,\mathbb{P}(x,r))\in\mathbb{P}\mathcal{V}_{c}^{1}\right.
\}.
\]

(ii) If there is a chain control set $E$ in $\mathbb{R}^{n}$ of the affine
control system (\ref{affine1}), the image $\mathbb{P}\left(  E\times
\{1\}\right)  $ in the projective Poincar\'{e} space $\mathbb{P}^{n}$ is
contained in $_{\mathbb{P}}E_{c}$.

(iii) If (\ref{bilinear0}) is uniformly hyperbolic, then there is a unique
chain control set $E$ in $\mathbb{R}^{n}$. It is compact and the central chain
control set is a compact subset of $\mathbb{P}^{n,1}$ satisfying
$_{\mathbb{P}}E_{c}$ $=\left\{  \mathbb{P}\left(  x,1\right)  \left\vert x\in
E\right.  \right\}  $. For every $u\in\mathcal{U}$ there exists a unique
element $x\in E$ with $\psi(t,x,u)\in E$ for all $t\in\mathbb{R}$.
\end{theorem}

In the nonhyperbolic case, the information given in this theorem on the chain
control sets of affine control systems is rather scarce: It is only shown that
their images on the projective Poincar\'{e} sphere are contained in the
central chain control set $_{\mathbb{P}}E_{c}$. A main goal of this paper is
to derive more information on chain controllability properties of affine
control systems.

\section{Noncompact spaces and strong chain controllability\label{Section3}}

We introduce a stronger version of chain control sets for noncompact spaces
and provide a characterization in terms of the control flow.

Consider a control system of the form (\ref{control_affine}) on a manifold $M$
with trajectories $\varphi(t,x,u),\allowbreak t\in\mathbb{R}$, and define the
set of jump functions by $\mathcal{P}(M):=\{\varepsilon:M\rightarrow
(0,\infty),$ continuous$\}$.

Define for $\varepsilon\in\mathcal{P}(M)$ a controlled $(\varepsilon,T)$-chain
from $x$ to $y$ by $T_{0},\ldots,T_{k-1}\allowbreak\geq T,u_{0},\ldots
,u_{k-1}\in\mathcal{U}$ and $x_{0}=x,x_{1},\ldots,x_{k}=y\in M$ with%
\[
d_{M}(\varphi(T_{j},x_{j},u_{j}),x_{j+1})<\varepsilon(\varphi(T_{j}%
,x_{j},u_{j}))\text{ for }j=0,\ldots,k-1.
\]
A point $x$ is strongly chain controllable to a point $y$ if for each
$\varepsilon\in\mathcal{P}(M)$ there is a controlled $(\varepsilon,T)$-chain
from $x$ to $y$. A nonvoid subset $F$ of $M$ is called strongly chain
controllable if every point $x\in F$ is strongly chain controllable to every
point $y\in F$.

\begin{definition}
\label{Definition_strong}For a control system on $M$ a nonvoid set $E^{\ast
}\subset M$ is a strong chain control set if (i) every $x\in E^{\ast}$ is
strongly chain controllable to every $y\in E^{\ast}$ (ii) for every $x\in
E^{\ast}$ there is $u\in\mathcal{U}$ with $\varphi(t,x,u)\in E^{\ast}$ for all
$t\in\mathbb{R}$, and (iii) $E^{\ast}$ is maximal with these properties.
\end{definition}

Note that jump functions $\varepsilon\in\mathcal{P}(M)$ may have values
arbitrarily close to $0$. In particular, if the boundary $\partial M$ of $M$
is nonvoid and $\partial M\cap M=\varnothing$, there are $\varepsilon
\in\mathcal{P}(M)$ with $\lim_{d_{M}(x,\partial M)\rightarrow0}\varepsilon
(x)=0$.

The following proposition presents some basic properties of strong chain
control sets. In particular, it shows that for compact state spaces chain
control sets and strong chain control sets coincide.

\begin{proposition}
\label{Proposition_compact}(i) Every strong chain control set $E^{\ast}$ is
closed, and if the intersection of strong chain control sets $E_{1}^{\ast}$
and $E_{2}^{\ast}$ is nonvoid, then $E_{1}^{\ast}=E_{2}^{\ast}$.

(ii) Every strong chain control set is contained in a chain control set.

(iii) Every compact chain control set $E$ is also a strong chain control set.

(iv) If $M$ is compact, then chain control sets and strong chain control sets coincide.
\end{proposition}

\begin{proof}
(i) Suppose that a sequence $\left(  x^{i}\right)  $ in $E^{\ast}$ converges
to $x\in M$. By compactness of $\mathcal{U}$ and continuity it follows that
there exists $u\in\mathcal{U}$ with $\varphi(t,x,u)\in\overline{E^{\ast}}$ for
all $t\in\mathbb{R}$.

Consider $y\in E^{\ast}$ and let $\varepsilon\in\mathcal{P}(M),T>0$. First, we
show that there exists a controlled $(\varepsilon,T)$-chain from $y$ to $x$.
There are $v^{i}\in\mathcal{U}$ with $\varphi(-T,x^{i},v^{i})\in E^{\ast}$.
Hence there are controlled $(\varepsilon,T)$-chains $\zeta^{i}$ from $y\in
E^{\ast}$ to $\varphi(-T,x^{i},v^{i})$ satisfying $x_{0}^{i}=y,x_{k_{i}}%
^{i}=\varphi(-T,x^{i},v^{i})$, and for $j=0,\ldots,k_{i}-1$%
\[
d_{M}(\varphi(T_{j}^{i},x_{j}^{i},u_{j}^{i}),x_{j+1}^{i})<\varepsilon
(\varphi(T_{j}^{i},x_{j}^{i},u_{j}^{i})).
\]
Extend each chain $\zeta^{i}$ by defining $T_{k_{i}}^{i}:=T,x_{k_{i}+1}%
^{i}:=x,$ and $u_{k_{i}}^{i}:=v^{i}(-T+\cdot)$. Now take $i$ large enough such
that $\varepsilon(x^{i})\geq\frac{1}{2}\varepsilon(x)$ and $d_{M}%
(x^{i},x)<\frac{1}{2}\varepsilon(x)$. It follows that $\varphi(T_{k_{i}}%
^{i},x_{k_{i}}^{i},u_{k_{i}}^{i})=x^{i}$ and%
\begin{align*}
d_{M}(\varphi(T_{k_{i}}^{i},x_{k_{i}}^{i},u_{k_{i}}^{i}),x_{k_{i}+1}^{i})  &
=d_{M}(\varphi(T,\varphi(-T,x^{i},v^{i}),v^{i}(-T+\cdot)),x)\\
&  =d_{M}(x^{i},x)<\frac{1}{2}\varepsilon(x)\leq\varepsilon(x^{i}%
)=\varepsilon(\varphi(T_{k_{i}}^{i},x_{k_{i}}^{i},u_{k_{i}}^{i})).
\end{align*}
Hence this defines a controlled $(\varepsilon,T)$-chain from $y$ to $x$.

Next we show that there is a controlled $(\varepsilon,T)$-chain from $x$ to
$y$. There is a controlled $(\varepsilon,2T)$-chain $\zeta^{i}$ from $x^{i}\in
E^{\ast}$ to $y$ satisfying%
\[
d_{M}(\varphi(T_{j}^{i},x_{j}^{i},u_{j}^{i}),x_{j+1}^{i})<\varepsilon
(\varphi(T_{j}^{i},x_{j}^{i},u_{j}^{i}))\text{ for }j=0,\ldots,k_{i}-1.
\]
Introduce an artificial jump point $z_{1}^{i}:=\varphi(T,x^{i},u_{0}^{i})$ at
time $S_{0}=T$, let $v_{0}^{i}:=u_{0}^{i}$, and define $S_{1}^{i}:=T_{0}%
^{i}-T\geq T$ and $v_{1}^{i}:=u_{0}^{i}(T+\cdot)$.

Then replace $T_{0}^{i},u_{0}^{i}$ by $S_{0},v_{0}^{i}$ and $S_{1}^{i}%
,v_{1}^{i}$. This yields again a controlled $(\varepsilon,T)$-chain from
$x^{i}\in E^{\ast}$ to $y$: In fact,%
\[
d_{M}(\varphi(S_{0},x^{i},v_{0}^{i}),z_{1}^{i})=d_{M}(\varphi(T,x^{i}%
,u_{0}^{i}),\varphi(T,x^{i},u_{0}^{i}))=0<\varepsilon(\varphi(S_{0}%
,x^{i},v_{0}^{i}))
\]
and $\varphi(S_{1}^{i},z_{1}^{i},v_{1}^{i})=\varphi(T_{0}^{i}-T,\varphi
(T,x^{i},u_{0}^{i}),u_{0}^{i}(T+\cdot))=\varphi(T_{0}^{i},x^{i},u_{0}^{i})$,
hence%
\[
d_{M}(\varphi(S_{1}^{i},z_{1}^{i},v_{1}^{i}),x_{2}^{i})<\varepsilon
(\varphi(S_{1}^{i},z_{1}^{i},v_{1}^{i})).
\]
We claim that $\inf_{j\in\mathbb{N}}\varepsilon(\varphi(T,x,u_{0}^{j}))>0$. In
fact otherwise there is a subsequence $u_{i_{j}}$ such that $\varepsilon
(\varphi(T,x,u_{0}^{i_{j}}))\rightarrow0$. By compactness of $\mathcal{U}$ and
continuity it follows that there is a subsequence converging to some element
$u\in\mathcal{U}$ with $\varepsilon\left(  \varphi(T,x,u)\right)  =0$ which
contradicts the definition of $\varepsilon$.

The claim together with continuity implies that one can take $i$ large enough
such that%
\[
d_{M}(\varphi(T,x,u_{0}^{i}),\varphi(T,x^{i},u_{0}^{i}))<\frac{1}{2}\inf
_{j\in\mathbb{N}}\varepsilon(\varphi(T,x,u_{0}^{j}))\leq\frac{1}{2}%
\varepsilon(\varphi(T,x,u_{0}^{i})).
\]
Taking, if necessary, a subsequence we may assume that $u_{0}^{i}$ converges
to an element $u\in\mathcal{U}$. Then
\[
\varepsilon(\varphi(T,x,u_{0}^{i}))\rightarrow\varepsilon(\varphi
(T,x,u))\text{ and }\varepsilon(\varphi(T,x^{i},u_{0}^{i}))\rightarrow
\varepsilon(\varphi(T,x,u)).
\]
Hence for $i$ large enough $\varepsilon(\varphi(T,x,u_{0}^{i}))<2\varepsilon
(\varphi(T,x^{i},u_{0}^{i}))$. Then we replace the starting point $x^{i}$ by
$x$ and find%
\begin{align*}
d_{M}(\varphi(S_{0},x,v_{0}^{i}),z_{1}^{i})  &  =d_{M}(\varphi(T,x,v_{0}%
^{i}),z_{1}^{i})\\
&  \leq d_{M}(\varphi(T,x,v_{0}^{i}),\varphi(T,x^{i},u_{0}^{i}))+d_{M}%
(\varphi(T,x^{i},u_{0}^{i}),z_{1}^{i})\\
&  =d_{M}(\varphi(T,x,u_{0}^{i}),\varphi(T,x^{i},u_{0}^{i}))+0\\
&  <\frac{1}{2}\varepsilon(\varphi(T,x,u_{0}^{i}))<\varepsilon(\varphi
(T,x,u_{0}^{i})).
\end{align*}
This shows that the modified chain is a controlled $(\varepsilon,T)$-chain
from $x$ to $y$.

Since $\varepsilon\in\mathcal{P}(M)$ and $T>0$ are arbitrary it follows that
$x\in E^{\ast}$. If the intersection of two chain control sets is nonvoid,
they coincide by the maximality property since the concatenation of controlled
$(\varepsilon,T)$-chains is a controlled $(\varepsilon,T)$-chain.

(ii) This follows by considering constant jump functions $\varepsilon
(x)\equiv\varepsilon>0$.

(iii) In view of assertion (ii) it suffices to show strong chain
controllability in $E$. Let $\varepsilon\in\mathcal{P}(M)$. Then
$\varepsilon(x)>0$ for all $x\in E$ and compactness of $E$ imply that
$\varepsilon_{0}:=\min\left\{  \varepsilon(x):x\in E\right\}  >0$. Thus any
controlled $(\varepsilon_{0},T)$-chain defines a strong controlled
$(\varepsilon,T)$-chain, hence strong chain controllability holds in $E$.

(iv) Here chain control sets and strong chain control sets are compact. Hence
the assertion follows by (ii) and (iii).
\end{proof}

From Example \ref{Example_strong} we can construct a linear control system,
where a strong chain control set is a proper subset of a chain control set and
contains a control set, see Figure \ref{fig1}.

\begin{example}
\label{Example_linear}Consider a linear control system given by%
\[
\left(
\begin{array}
[c]{c}%
\dot{x}(t)\\
\dot{y}(t)\\
\dot{z}(t)
\end{array}
\right)  =\left(
\begin{array}
[c]{ccc}%
0 & 1 & 0\\
0 & 0 & 0\\
0 & 0 & 1
\end{array}
\right)  \left(
\begin{array}
[c]{c}%
x(t)\\
y(t)\\
z(t)
\end{array}
\right)  +\left(
\begin{array}
[c]{c}%
0\\
0\\
1
\end{array}
\right)  u(t)\text{ with }u(t)\in\Omega=[-1,1].
\]
By Example \ref{Example_strong} it follows for the subspace $\mathbb{R}%
^{2}\times\{0\}$ that the strong chain recurrent set is $\mathbb{R}%
\times\{0\}\times\{0\}$ (here the controls $u$ do not act). The control set
(i.e., the maximal set of approximate controllability) containing
$0\in\mathbb{R}^{3}$ is $D_{0}=\{(0,0)\}\times(-1,1)$. Since in the interior
of $D_{0}$ with respect to $\{(0,0)\}\times\mathbb{R}$ exact controllability
holds, the strong chain control set is%
\[
E^{\ast}=\mathbb{R}\times\{0\}\times\lbrack-1,1].
\]
On the other hand, $L^{0}=\mathbb{R}^{2}\times\{0\}$ is the center subspace of
the homogeneous (uncontrolled) part, and Colonius, Santana, and Viscovini
\cite[Theorem 4.8]{ColSV24} implies that the chain control set is%
\[
E=\overline{D_{0}}+L^{0}=\left(  \{(0,0)\}\times\lbrack-1,1]\right)  +\left(
\mathbb{R}^{2}\times\{0\}\right)  =\mathbb{R}^{2}\times\lbrack-1,1]\not =%
E^{\ast}.
\]

\end{example}

The next theorem clarifies the relation between strong chain control sets and
maximal invariant strongly chain transitive sets of the control flow. The
proof is a modification of the one for chain control sets; cf. Colonius and
Kliemann \cite[Theorem 4.3.11]{ColK00} or Kawan \cite[Proposition
1.24(iv)]{Kawa13}.

\begin{theorem}
\label{Theorem_equivalence}Consider a control system of the form
(\ref{control_affine}) on $M$ with control flow $\Phi$ on $\mathcal{U}\times
M$.

(i) If $E^{\ast}\subset M$ is a strong chain control set, then%
\begin{equation}
\mathcal{E}^{\ast}:=\{(u,x)\in\mathcal{U}\times M\left\vert \varphi(t,x,u)\in
E^{\ast}\text{ for all }t\in\mathbb{R}\right.  \} \label{lift_E}%
\end{equation}
is a maximal invariant strongly chain transitive set for the control flow
$\Phi$.

(ii) Conversely, let $\mathcal{E}^{\ast}\subset\mathcal{U}\times M$ be a
maximal invariant strongly chain transitive set for the control flow $\Phi$.
Then
\begin{equation}
E^{\ast}:=\left\{  x\in M\left\vert \exists u\in\mathcal{U}:(u,x)\in
\mathcal{E}^{\ast}\right.  \right\}  \label{down_E}%
\end{equation}
is a strong chain control set.
\end{theorem}

\begin{proof}
(i) Let $E^{\ast}$ be a strong chain control set and consider
$(u,x),\;(v,y)\in\mathcal{E}^{\ast}$ defined by (\ref{lift_E}). Pick $T>0$ and
$\varepsilon\in\mathcal{P}(\mathcal{U}\times M)$.

\textbf{Claim 1:} Define $\varepsilon_{M}(x):=\min\left\{  \varepsilon
(u,x)\left\vert u\in\mathcal{U}\right.  \right\}  ,x\in M.$ Then
$\varepsilon_{M}\in\mathcal{P}(M)$.

By compactness of $\mathcal{U}$ it is clear that $\varepsilon_{M}(x)>0$ for
all $x\in M$ and it remains to show that $\varepsilon_{M}$ is continuous.
Consider a sequence $x_{i}\rightarrow x$. It follows that $\varepsilon
(u,x_{i})\geq\varepsilon_{M}(x_{i})$ for all $u\in\mathcal{U}$. Continuity of
$\varepsilon$ implies that $\varepsilon(u,x)\geq\lim\sup_{i\rightarrow\infty
}\varepsilon_{M}(x_{i})$ for all $u\in\mathcal{U}$ and hence $\varepsilon
_{M}(x)\geq\lim\sup_{i\rightarrow\infty}\varepsilon_{M}(x_{i})$.

By compactness of $\mathcal{U}$ there are $u_{i}\in\mathcal{U}$ with
$f(u_{i},x_{i})=\varepsilon_{M}(x_{i})$ and there is a subsequence $(u_{i_{k}%
})$ converging to some $u^{\ast}\in\mathcal{U}$. It follows that
$\varepsilon_{M}(x_{i_{k}})=\varepsilon(u_{i_{k}},x_{i_{k}})\rightarrow
\varepsilon(u^{\ast},x)\geq\varepsilon_{M}(x)$. Since this applies to any
subsequence of $(u_{i})$ it follows that $\varepsilon_{M}(x)\leq\lim
\inf_{i\rightarrow\infty}\varepsilon_{M}(x_{i})$ proving \textbf{Claim 1}.

Strong chain controllability from $\varphi(2T,x,u)\in E^{\ast}$ to
$\varphi(-T,y,v)\in E^{\ast}$ yields the existence of a controlled
$(\varepsilon_{M},T)$-chain given by $x_{0}=\varphi(2T,x,u),x_{1},\ldots
,x_{k}=\varphi(-T,y,v)$ in $M,\,u_{0},\ldots,u_{k-1}\in\mathcal{U}%
,\,T_{0},\ldots,T_{k-1}\geq T$ with
\begin{equation}
d_{M}(\varphi(T_{j},x_{j},u_{j}),x_{j+1})<\varepsilon_{M}(\varphi(T_{j}%
,x_{j},u_{j}))\text{\ for }j=0,\ldots,k-1. \label{r1}%
\end{equation}
We construct an $(\varepsilon,T)$-chain from $(u,x)$ to $(v,y)$ in the
following way. Define
\[%
\begin{array}
[c]{cccc}%
T_{-2}=T, & x_{-2}=x, & v_{-2}=u, & \\
T_{-1}=T, & x_{-1}=\varphi(T,x,u), & v_{-1}(t)= & \left\{
\begin{array}
[c]{cc}%
u(T_{-2}+t) & \text{for }t\leq T_{-1}\\
u_{0}(t-T_{-1}) & \text{for }t>T_{-1}%
\end{array}
\right.
\end{array}
\]
and let the times $T_{0},\ldots,T_{k-1}$ and the points $x_{0},\ldots,x_{k}$
be as given earlier. Furthermore, set
\[%
\begin{array}
[c]{ccc}%
T_{k}=T, & x_{k+1}=y, & v_{k+1}=v,
\end{array}
\]
and define for $j=0,\ldots,k-2$%
\begin{align*}
v_{j}(t)  &  =\left\{
\begin{array}
[c]{lll}%
v_{j-1}(T_{j-1}+t) & \text{for} & t\leq0\\
u_{j}(t) & \text{for} & 0<t\leq T_{j}\\
u_{j+1}(t-T_{j}) & \text{for} & t>T_{j},
\end{array}
\right. \\
v_{k-1}(t)  &  =\left\{
\begin{array}
[c]{lll}%
v_{k-2}(T_{k-2}+t) & \text{for} & t\leq0\\
u_{k-1}(t) & \text{for} & 0<t\leq T_{k-1}\\
v(t-T_{k-1}-T) & \text{for} & t>T_{k-1},
\end{array}
\right. \\
v_{k}(t)  &  =\left\{
\begin{array}
[c]{ll}%
v_{k-1}(T_{k-1}+t) & \text{for }t\leq0\\
v(t-T) & \text{for }t>0.
\end{array}
\right.
\end{align*}
Define a constant $\widehat{\varepsilon}$ by%
\[
\widehat{\varepsilon}:=\frac{1}{2}\min_{j=-2,\ldots,k}\varepsilon_{M}%
(\varphi(T_{j},x_{j},v_{j})).
\]
Recall the definition of the metric $d_{\mathcal{U}}$ on $\mathcal{U}$ given
in (\ref{metric_U}) and choose $N\in\mathbb{N}$ large enough such that
$\sum_{i=N+1}^{\infty}2^{-i}<\widehat{\varepsilon}$. There exists $S>0$ such
that%
\[
\int_{\mathbb{R}\setminus\left[  -S,S\right]  }\left\vert y_{i}(\tau
)\right\vert \,\,d\tau<\widehat{\varepsilon}/\mathrm{\,diam}\,U\text{ for
}i=1,\ldots,N.
\]
We can assume without loss of generality that $T\geq S$. Since $\varepsilon
\in\mathcal{P}(\mathcal{U}\times M)$ and $T\geq S$ are arbitrary, the
following claim implies that $\mathcal{E}$ is strongly chain transitive.

\textbf{Claim 2: }An $(\varepsilon,T)$-chain from $(u,x)$ to $(v,y)$ is given
by%
\[
(v_{-2},x_{-2}),\,(v_{-1},x_{-1}),\ldots,(v_{k+1},x_{k+1})\text{ in
}\mathcal{U}\times M\text{ and }T_{-2},\,T_{-1},\ldots,T_{k}\geq T.
\]
For the proof of the claim observe first that there are no jumps at
$x_{-1},x_{0}$, and $x_{k+1}$ and%
\[
\varphi(T_{j},x_{j},u_{j})=\varphi(T_{j},x_{j},v_{j})\text{ for }%
j=0,\ldots,k-1.
\]
Hence (\ref{r1}) yields for $j=-2,\ldots,k$
\begin{align*}
d_{M}(\varphi(T_{j},x_{j},v_{j}),x_{j+1})  &  <\varepsilon_{M}(\varphi
(T_{j},x_{j},v_{j}))\\
&  \leq\varepsilon(v_{j}(T_{j}+\cdot),\varphi(T_{j},x_{j},v_{j}))=\varepsilon
(\Phi_{T_{j}}(v_{j},x_{j})).
\end{align*}
Concerning the component in $\mathcal{U}$ we have to show that%
\begin{equation}
d_{\mathcal{U}}(v_{j}(T_{j}+\cdot),v_{j+1})<\varepsilon(\Phi_{T_{j}}%
(v_{j},x_{j}))\text{ for }j=-2,\,-1,\ldots,k. \label{controls}%
\end{equation}
By choice of $T$ and $N$ all $w_{1},\,w_{2}\in\mathcal{U}$ satisfy%
\begin{align*}
&  d_{\mathcal{U}}(w_{1},w_{2})=\sum_{i=1}^{\infty}\frac{1}{2^{i}}%
\frac{\left\vert \int_{\mathbb{R}}\left\langle w_{1}(t)-w_{2}(t),y_{i}%
(t)\right\rangle \,dt\right\vert }{1+\left\vert \int_{\mathbb{R}}\left\langle
w_{1}(t)-w_{2}(t),y_{i}(t)\right\rangle \,dt\right\vert }\\
&  \leq\sum_{i=1}^{N}\frac{1}{2^{i}}\left\{  \left\vert \int_{\mathbb{R}%
\setminus\left[  -T,T\right]  }\left\langle w_{1}(t)-w_{2}(t),y_{i}%
(t)\right\rangle \,dt\right\vert \right.  +\left.  \left\vert \int_{-T}%
^{T}\left\langle w_{1}(t)-w_{2}(t),y_{i}(t)\right\rangle \,dt\right\vert
\right\}  +\widehat{\varepsilon}\\
&  <2\widehat{\varepsilon}+\max_{i=1,\ldots,N}\int_{-T}^{T}\left\vert
w_{1}(t)-w_{2}(t)\right\vert \,\left\vert y_{i}(t)\right\vert \,dt.
\end{align*}
For all considered pairs of control functions in (\ref{controls}) the
integrands vanish, hence the distances are bounded by $2\widehat{\varepsilon}%
$, and for $j=0,\ldots,k-1$ the definition of $\widehat{\varepsilon}$ shows
that%
\[
2\widehat{\varepsilon}=\min_{j=-2,\ldots,k}\varepsilon_{M}(\varphi(T_{j}%
,x_{j},v_{j}))\leq\min_{j=-2,\ldots,k}\varepsilon(\Phi_{T_{j}}(v_{j},x_{j})).
\]
Thus the claim in (\ref{controls}) holds.

(ii) Let $\mathcal{E}^{\ast}$ be a maximal invariant strongly chain transitive
set in $\mathcal{U}\times M$. Suppose that for $x,y\in M$ there are
$u,v\in\mathcal{U}$ with $(u,x),(v,y)\allowbreak\in E^{\ast}$, hence
$\varphi(t,x,u),\varphi(t,y,v)\allowbreak\in E^{\ast}$ defined by
(\ref{down_E}) for all $t\in\mathbb{R}$. Choose $\varepsilon_{M}\in
\mathcal{P}(M)$ and $T>0$ and define $\varepsilon\in\mathcal{P}(\mathcal{U}%
\times M)$ by $\varepsilon(u,x):=\varepsilon_{M}(x),u\in\mathcal{U},x\in M$. A
corresponding $(\varepsilon,T)$-chain in $\mathcal{U}\times M$ for the control
flow $\Phi$ yields $x_{j},u_{j},T_{j}$ such that the corresponding
trajectories satisfy%
\begin{align*}
d_{M}(\varphi(T_{j},x_{j},u_{j}),x_{j+1})  &  \leq d(\Phi_{T_{j}}(x_{j}%
,u_{j}),(u_{j+1},x_{j+1}))<\varepsilon(\Phi_{T_{j}}(x_{j},u_{j}))\\
&  =\varepsilon_{M}(\varphi(T_{j},x_{j},u_{j})).
\end{align*}
This shows that $E^{\ast}$ defined by (\ref{down_E}) is strongly chain controllable.

The proof of (i) and (ii) is concluded by the observation that $E^{\ast}$ is a
maximal invariant strongly chain controllable set if and only if
$\mathcal{E}^{\ast}$ is a maximal invariant strongly chain transitive set.
\end{proof}

As an easy consequence, we find that conjugacies of control systems (cf.
Definition \ref{Definition_conjugacy}) preserve strong chain control sets.

\begin{corollary}
Suppose that two control systems of the form (\ref{control_con}) on $M$ and
$N$, respectively, are topologically conjugate by a homeomorphism
$H:M\rightarrow N$. Then $E^{\ast}$ is a strong chain control set in $M$ if
and only if $H(E^{\ast})$ is a strong chain control set in $N$.
\end{corollary}

\begin{proof}
The homeomorphism $H$ can be extended to a topological conjugacy
$\mathrm{id}_{\mathcal{U}}\times H$ of the associated control flows $\Phi$ on
$\mathcal{U}\times M$ and $\Psi$ on $\mathcal{U}\times N$. By Theorem
\ref{Theorem_equivalence} the set $E^{\ast}$ is a strong chain control set if
and only if the lift $\mathcal{E}^{\ast}$ is a maximal invariant strongly
chain transitive set for $\Phi$. Theorem \ref{Theorem_components} shows that
it is mapped onto the maximal invariant strongly chain transitive set $\left(
\mathrm{id}_{\mathcal{U}}\times H\right)  (\mathcal{E}^{\ast})$, which again
by Theorem \ref{Theorem_equivalence} is the lift of a strong chain control set
which has the form $H(E^{\ast})$. The converse follows by considering $H^{-1}$.
\end{proof}

\begin{remark}
An alternative proof of this corollary uses Hurley's lemma, Lemma
\ref{Lemma2_Hurley}, directly for topologically conjugate control systems.
\end{remark}

\section{Strong chain control sets of affine systems\label{Section4}}

In this section, we consider affine control systems of the form (\ref{affine1}%
) on $\mathbb{R}^{n}$ and study the induced control systems on projective
space $\mathbb{P}^{n}$ and on the northern hemisphere $\mathbb{S}^{n,+}$ of
the Poincar\'{e} sphere,%
\[
\mathbb{S}^{n,+}:=\{s=(s_{1},\ldots,s_{n+1})\in\mathbb{S}^{n}\left\vert
{}\right.  \sum_{i=1}^{n+1}s_{i}^{2}=1,\left\vert s_{i}\right\vert
\leq1,i=1,\ldots,n+1,s_{n+1}>0\}.
\]
Define a function $h:\mathbb{R}^{n}\rightarrow\mathbb{S}^{n,+}$ by%
\begin{equation}
h(x)=h(x_{1},\ldots,x_{n}):=\left(  \frac{x_{1}}{\left\Vert (x,1)\right\Vert
},\ldots,\frac{x_{n}}{\left\Vert (x,1)\right\Vert },\frac{1}{\left\Vert
(x,1)\right\Vert }\right)  =\frac{(x,1)}{\left\Vert (x,1)\right\Vert }.
\label{h}%
\end{equation}
Note that $\left\Vert x\right\Vert \rightarrow\infty$ if and only if $h(x)$
approaches the equator%
\[
\mathbb{S}^{n,0}:=\{s=(s_{1},\ldots,s_{n+1})\in\mathbb{S}^{n}:s_{n+1}=0\}.
\]
Define a control system on $\mathbb{S}^{n}$ by the projection of the bilinear
control system (\ref{hom_d+1}); cf. Colonius and Kliemann \cite[Section
6.1]{ColK00}. With%
\[
A_{i}^{\prime}:=\left(
\begin{array}
[c]{cc}%
A_{i} & a_{i}\\
0 & 0
\end{array}
\right)  \text{ for }i=0,\ldots,m,
\]
the projected system is given by%
\begin{equation}
\dot{s}=\left[  A_{0}^{\prime}-s^{T}A_{0}^{\prime}s\cdot I\right]
s+\sum_{i=1}^{m}u_{i}(t)\left[  A_{i}^{\prime}-s^{T}A_{i}^{\prime}s\cdot
I\right]  s,u\in\mathcal{U}. \label{system_S}%
\end{equation}
It is clear that the control system (\ref{system_S}) also induces a control
system on projective space $\mathbb{P}^{n}$. The trajectories on
$\mathbb{S}^{n}$ are denoted by%
\[
\mathbb{S}\psi^{1}(t,s,u)=\frac{\psi^{1}(t,x,1,u)}{\left\Vert \psi
^{1}(t,x,1,u)\right\Vert }=\frac{\left(  \psi(t,x,u),1\right)  }{\left\Vert
\left(  \psi(t,x,u),1\right)  \right\Vert }\text{ with }s=\frac{(x,1)}%
{\left\Vert (x,1)\right\Vert }.
\]
This control-affine system can be restricted to the northern hemisphere with
trajectories $\mathbb{S}\psi^{1,+}(t,s,u),t\in\mathbb{R}$. We denote the
corresponding control flow by%
\[
\mathbb{S}\Psi^{1,+}:\mathbb{R}\times\mathcal{U}\times\mathbb{S}%
^{n,+}\rightarrow\mathcal{U}\times\mathbb{S}^{n,+},~\mathbb{S}\Psi
^{1,+}(t,u,s):=(u(t+\cdot),\mathbb{S}\psi^{1,+}(t,s,u)),
\]
for $(t,u,s)\in\mathbb{R}\times\mathcal{U}\times\mathbb{S}^{n,+}$.

\begin{remark}
\label{Remark_southern}The elements of the southern hemisphere $\mathbb{S}%
^{n,-}$ satisfy $s_{n+1}<0$ and one obtains a map $h_{-}:\mathbb{R}%
^{n}\rightarrow\mathbb{S}^{n,-}$ by%
\[
h_{-}(x)=h_{-}(x_{1},\ldots,x_{n}):=\left(  \frac{-x_{1}}{\left\Vert
(x,1)\right\Vert },\ldots,\frac{-x_{n}}{\left\Vert (x,1)\right\Vert }%
,\frac{-1}{\left\Vert (x,1)\right\Vert }\right)  .
\]
Since $\psi^{1}$ is linear there is a control system on $\mathbb{S}^{n,-}$
with trajectories%
\[
-\mathbb{S}\psi^{1}(t,s,u)=\frac{-\psi^{1}(t,x,1,u)}{\left\Vert \psi
^{1}(t,x,1,u)\right\Vert }=\frac{\psi^{1}(t,-x,-1,u)}{\left\Vert \psi
^{1}(t,x,1,u)\right\Vert }=\mathbb{S}\psi^{1}(t,-s,u).
\]
This shows that the control systems on the northern and the southern
hemispheres are conjugate.
\end{remark}

In the following, we analyze the relation between the original control system
on $\mathbb{R}^{n}$ and the induced control system on $\mathbb{S}^{n,+}$. A
metric on $\mathbb{S}^{n,+}$ is given by%
\[
d(s,s^{\prime})=\max_{i=1,\ldots,n+1}\left\vert s_{i}-s_{i}^{\prime
}\right\vert \text{ for }s=(s_{1},\ldots,s_{n+1}),s^{\prime}=(s_{1}^{\prime
},\ldots,s_{n+1}^{\prime})\in\mathbb{S}^{n,+}.
\]
On $\mathbb{R}^{n}$ we consider the maximum norm $\left\Vert \cdot\right\Vert
_{\infty}$.

\begin{lemma}
\label{Lemma_h}The map $h:\mathbb{R}^{n}\rightarrow\mathbb{S}^{n,+}$ is a
homeomorphism with inverse%
\[
h^{-1}:\mathbb{S}^{n,+}\rightarrow\mathbb{R}^{n},h^{-1}(s)=h^{-1}(s_{1}%
,\ldots,s_{n+1})=\left(  \frac{s_{1}}{s_{n+1}},\ldots,\frac{s_{n}}{s_{n+1}%
}\right)  .
\]

\end{lemma}

\begin{proof}
The map $h$ is injective since $h(x)=h(y)$ implies $\frac{1}{\left\Vert
(x,1)\right\Vert }=\frac{1}{\left\Vert (y,1)\right\Vert }$ and%
\[
\frac{x_{i}}{\left\Vert (x,1)\right\Vert }=\frac{y_{i}}{\left\Vert
(y,1)\right\Vert }\text{, i.e., }x_{i}=y_{i}\text{ for }i=1,\ldots,n.
\]
Surjectivity follows since for $s=(s_{1},\ldots,s_{n+1})\in\mathbb{S}^{n,+}$
the point $x=\left(  \frac{s_{1}}{s_{n+1}},\ldots,\frac{s_{n}}{s_{n+1}%
}\right)  $ satisfies%
\[
\left\Vert (x,1)\right\Vert ^{2}=\frac{s_{1}^{2}}{s_{n+1}^{2}}+\cdots
+\frac{s_{n}^{2}}{s_{n+1}^{2}}+1=\frac{1}{s_{n+1}^{2}},
\]
hence
\[
h\left(  \frac{s_{1}}{s_{n+1}},\ldots,\frac{s_{n}}{s_{n+1}}\right)  =\left(
s_{1},\ldots,s_{n},s_{n+1}\right)  .
\]
This also implies the formula for $h^{-1}$. The map $h$ is continuous by
Colonius and Santana \cite[Proposition 9(iii)]{ColS24a}, which shows that
\[
d(h(x),h(y))\leq2\left\Vert x-y\right\Vert _{\infty}\text{ for }%
x,y\in\mathbb{R}^{n}.
\]
It remains to prove the continuity of $h^{-1}$. First note that for $s,s^{\prime
}\in\mathbb{S}^{n,+}$ we find%
\begin{align*}
\left\Vert h^{-1}(s)-h^{-1}(s^{\prime})\right\Vert _{\infty}  &  =\left\Vert
\left(  \frac{s_{1}}{s_{n+1}},\ldots,\frac{s_{n}}{s_{n+1}}\right)  -\left(
\frac{s_{1}^{\prime}}{s_{n+1}^{\prime}},\ldots,\frac{s_{n}^{\prime}}%
{s_{n+1}^{\prime}}\right)  \right\Vert _{\infty}\\
&  =\max_{i=1,\ldots,n}\left\vert \frac{s_{i}}{s_{n+1}}-\frac{s_{i}^{\prime}%
}{s_{n+1}^{\prime}}\right\vert .
\end{align*}
Now let $s^{k}$ converge to $s$ in $\mathbb{S}^{n,+}$, hence $d(s^{k}%
,s)=\max_{i=1,\ldots,n+1}\left\vert s_{i}^{k}-s_{i}\right\vert \rightarrow0$.
Then it follows that%
\begin{align*}
\left\Vert h^{-1}(s^{k})-h^{-1}(s)\right\Vert _{\infty}  &  =\max
_{i=1,\ldots,n}\left\vert \frac{s_{i}^{k}}{s_{n+1}^{k}}-\frac{s_{i}}{s_{n+1}%
}\right\vert \\
&  \leq\max_{i=1,\ldots,n}\left\vert \frac{s_{i}^{k}}{s_{n+1}^{k}}-\frac
{s_{i}}{s_{n+1}^{k}}\right\vert +\max_{i=1,\ldots,n}\left\vert \frac{s_{i}%
}{s_{n+1}^{k}}-\frac{s_{i}}{s_{n+1}}\right\vert \\
&  \leq\frac{1}{\left\vert s_{n+1}^{k}\right\vert }\max_{i=1,\ldots
,n}\left\vert s_{i}^{k}-s_{i}\right\vert +\left\vert \frac{1}{s_{n+1}^{k}%
}-\frac{1}{s_{n+1}}\right\vert \max_{i=1,\ldots,n}\left\vert s_{i}\right\vert
\\
&  \leq\frac{1}{\left\vert s_{n+1}^{k}\right\vert }\max_{i=1,\ldots
,n}\left\vert s_{i}^{k}-s_{i}\right\vert +\left\vert \frac{1}{s_{n+1}^{k}%
}-\frac{1}{s_{n+1}}\right\vert .
\end{align*}
The right hand side converges to $0$ for $k\rightarrow\infty$, since
$s_{i}^{k}\rightarrow s_{i}$ for $i=1,\ldots,n+1$.
\end{proof}

The next result shows that the control flows on $\mathcal{U}\times
\mathbb{R}^{n}$ and on $\mathcal{U}\times\mathbb{S}^{n,+}$ are topologically conjugate.

\begin{theorem}
\label{Theorem_conjugate}Consider the affine control flow $\Psi$ on
$\mathcal{U}\times\mathbb{R}^{n}$ corresponding to the affine control system
(\ref{affine1}) and the control flow $\mathbb{S}\Psi^{1,+}$ on $\mathcal{U}%
\times\mathbb{S}^{n,+}$ corresponding to the induced control system
(\ref{system_S}) on $\mathbb{S}^{n,+}$.

(i) The map $\mathrm{id}_{\mathcal{U}}\times h:\mathcal{U}\times\mathbb{R}%
^{n}\rightarrow\mathcal{U}\times\mathbb{S}^{n,+}$ is a topological conjugation
of the flows $\Psi$ and $\mathbb{S}\Psi^{1,+}$, hence%
\[
\left(  \mathrm{id}_{\mathcal{U}}\times h\right)  \Psi_{t}(u,x)=\mathbb{S}%
\Psi_{t}^{1,+}(u,h(x))\text{ for }t\in\mathbb{R},(u,x)\in\mathcal{U}%
\times\mathbb{R}^{n}.
\]

(ii) The maximal invariant strongly chain transitive sets for $\Psi$ are
mapped by $\mathrm{id}_{\mathcal{U}}\times h$ onto the maximal invariant
strongly chain transitive sets for $\mathbb{S}\Psi_{t}^{1,+}$.
\end{theorem}

\begin{proof}
(i) By Lemma \ref{Lemma_h} the map $\mathrm{id}_{\mathcal{U}}\times h$ is a
homeomorphism and it remains to prove the conjugation equality. For the
component in $\mathbb{R}^{n}$ one computes%
\begin{align*}
h(\psi(t,x,u))  &  =\left(  \frac{\psi_{1}(t,x,u)}{\left\Vert (\psi
(t,x,u),1)\right\Vert },\ldots,\frac{\psi_{n}(t,x,u)}{\left\Vert
(\psi(t,x,u),1)\right\Vert },\frac{1}{\left\Vert (\psi(t,x,u),1)\right\Vert
}\right) \\
&  =\frac{\psi^{1}(t,x,1,u)}{\left\Vert \psi^{1}(t,x,1,u)\right\Vert }%
=\frac{\psi^{1}\left(  t,\frac{\left(  x,1\right)  }{\left\Vert
(x,1)\right\Vert },u\right)  }{\left\Vert \psi^{1}\left(  t,\frac{\left(
x,1\right)  }{\left\Vert (x,1)\right\Vert },u\right)  \right\Vert }\\
&  =\mathbb{S}\psi^{1,+}(t,h(x),u).
\end{align*}
This implies the conjugation equality%
\begin{align*}
(\mathrm{id}_{\mathcal{U}}\times h)\Psi_{t}(u,x)  &  =\left(  u(t+\cdot
),h(\psi(t,x,u))\right)  =\left(  u(t+\cdot),\mathbb{S}\psi^{1,+}%
(t,h(x),u)\right) \\
&  =\mathbb{S}\Psi_{t}^{1,+}(u,h(x)).
\end{align*}

(ii) This follows by (i) and Theorem \ref{Theorem_components}.
\end{proof}

For the strong chain control sets of the affine control system on
$\mathbb{R}^{n}$ and the induced control system on $\mathbb{S}^{n,+}$ we
obtain the following consequence.

\begin{corollary}
\label{Corollary_relation}Consider the affine control system (\ref{affine1})
on $\mathbb{R}^{n}$ and the induced control system (\ref{system_S}) on the
upper hemisphere $\mathbb{S}^{n,+}$ of the Poincar\'{e} sphere. The strong
chain control sets on $\mathbb{S}^{n,+}$ have the form $_{\mathbb{S}}E^{\ast
}=h(E^{\ast})$, where the $E^{\ast}$ are the strong chain control sets on
$\mathbb{R}^{n}$.
\end{corollary}

\begin{proof}
Theorem \ref{Theorem_conjugate} shows that the maximal invariant strongly
chain transitive sets on $\mathcal{U}\times\mathbb{R}^{n}$ for $\Psi$ are
mapped by $\mathrm{id}_{\mathcal{U}}\times h$ onto the maximal invariant
strongly chain transitive sets on $\mathcal{U}\times\mathbb{S}^{n,+}$ for
$\mathbb{S}\Psi_{t}^{1,+}$. By Theorem \ref{Theorem_equivalence} the strong
chain control sets uniquely correspond to the maximal invariant strongly chain
transitive sets of the control flows. Thus the assertion follows.
\end{proof}

Next we briefly indicate the relation between the affine control system
(\ref{affine1}) on $\mathbb{R}^{n}$ and the induced control system on
$\mathbb{P}^{n}$. Recall from Subsection \ref{Subsection2.2} that
$\mathbb{P}^{n,1}=\{\mathbb{P}(x,1)\in\mathbb{P}^{n}\left\vert x\in
\mathbb{R}^{n}\right.  \}$.

\begin{theorem}
\label{Theorem_projective}Consider the affine control flow $\Psi$ on
$\mathcal{U}\times\mathbb{R}^{n}$ corresponding to the affine control system
(\ref{affine1}) and the control flow $\mathbb{P}\Psi^{1}$ on $\mathcal{U}%
\times\mathbb{P}^{n}$ corresponding to the control system induced by
(\ref{hom_d+1}) on $\mathbb{P}^{n}$. Let $\mathbb{P}h:\mathbb{R}%
^{n}\rightarrow\mathbb{P}^{n,1}$ be the map $\mathbb{P}h(x):=\mathbb{P}%
(x,1),x\in\mathbb{R}^{n}$.

(i) The map $\mathrm{id}_{\mathcal{U}}\times\mathbb{P}h:\mathcal{U}%
\times\mathbb{R}^{n}\rightarrow\mathcal{U}\times\mathbb{P}^{n,1}$ is a
topological conjugation of the flows $\Psi$ and $\mathbb{P}\Psi^{1}$
restricted to $\mathbb{P}^{n,1}$, hence%
\[
\left(  \mathrm{id}_{\mathcal{U}}\times\mathbb{P}h\right)  \Psi_{t}%
(u,x)=\mathbb{P}\Psi_{t}^{1}(u,\mathbb{P}h(x))\text{ for }t\in\mathbb{R}%
,(u,x)\in\mathcal{U}\times\mathbb{R}^{n}.
\]

(ii) The maximal invariant strongly chain transitive sets for $\Psi$ are
mapped by $\mathrm{id}_{\mathcal{U}}\times\mathbb{P}h$ onto the maximal
invariant strongly chain transitive sets for $\mathbb{P}\Psi^{1}$ restricted
to $\mathcal{U}\times\mathbb{P}^{n,1}$.

(iii) The strong chain control sets on $\mathbb{P}^{n,1}$ have the form
$_{\mathbb{P}}E^{\ast}=\mathbb{P}h(E^{\ast})$, where the $E^{\ast}$ are the
strong chain control sets on $\mathbb{R}^{n}$.
\end{theorem}

\begin{proof}
We only show that the map $\mathbb{P}h$ is a homeomorphism. Then the theorem
follows along the same lines as the proofs of Theorem \ref{Theorem_conjugate}
and Corollary \ref{Corollary_relation}.

Recall that $\mathbb{P}^{n}=(\mathbb{R}^{n+1}\setminus\{0\})/\thicksim$, where
$\thicksim$ is the equivalence relation $x\thicksim y$ if $y=\lambda x$ with
some $\lambda\not =0$. An atlas of $\mathbb{P}^{n}$ is given by $n+1$ charts
$(U_{i},\alpha_{i})$, where $U_{i}$ is the set of equivalence classes
$[x_{1}:\cdots:x_{n+1}]$ with $x_{i}\not =0$ (using homogeneous coordinates)
and $\alpha_{i}:U_{i}\rightarrow\mathbb{R}^{n}$ is defined by
\[
\alpha_{i}([x_{1}:\cdots:x_{n+1}])=\left(  \frac{x_{1}}{x_{i}},\ldots
,\widehat{\frac{x_{i}}{x_{i}}},\ldots,\frac{x_{n+1}}{x_{i}}\right)  ;
\]
here the hat means that the $i$-th entry is omitted. In homogeneous
coordinates $\mathbb{P}^{n,1}$ is described by
\[
\mathbb{P}^{n,1}=\left\{  [x_{1}:\cdots:x_{n}:1]\left\vert (x_{1},\ldots
,x_{n})\in\mathbb{R}^{n}\right.  \right\}  .
\]
Note that $U_{n+1}=\mathbb{P}^{n,1}$ and $\mathbb{P}h(x)=\left(  \alpha
_{n+1}\right)  ^{-1}(x),x\in\mathbb{R}^{n}$. A trivial atlas for
$\mathbb{P}^{n,1}$ is given by $\left\{  (U_{n+1},\alpha_{n+1})\right\}  $
proving that $\mathbb{P}^{n,1}$ is a manifold which under $\left(
\mathbb{P}h\right)  ^{-1}$ is diffeomorphic to $\mathbb{R}^{n}$.
\end{proof}

Finally, we analyze when strong chain control sets in $\mathbb{R}^{n}$ are unbounded.

\begin{proposition}
\label{Proposition_unbounded}Let $E^{\ast}$ be a strong chain control set of
the affine system (\ref{affine1}). Then $E^{\ast}$ is unbounded if there are
$u\in\mathcal{U}$ and $\tau>0$ with $\psi(\tau,x,u)=x\in E^{\ast}$ such that
the fundamental solution $X_{u}(t,s),t,s\in\mathbb{R}$, of the homogeneous
part (\ref{h}) satisfies $1\in\mathrm{spec}(X_{u}(\tau,0))$.
\end{proposition}

\begin{proof}
Extend $u_{\left\vert [0,\tau]\right.  }$ to a $\tau$-periodic control also
denoted by $u$. By the variation-of-parameters formula, this yields the $\tau
$-periodic trajectory $\psi(t,x,u),t\in\mathbb{R}$, satisfying%
\[
\int_{0}^{\tau}X_{u}(\tau,s)\left(  a_{0}+\sum_{i=1}^{m}u_{i}(s)a_{i}\right)
ds=(I-X_{u}(\tau,0))x.
\]
The assumption means that the eigenspace $\mathbf{E}(X_{u}(\tau,0);1))$ is
nontrivial. It follows that for every $y$ in the affine subspace
$Y:=x+\mathbf{E}(X_{u}(\tau,0);1))$ there is a $\tau$-periodic solution of
(\ref{affine1}) with control $u$ starting in $y$. This implies that
$Y\subset\mathbb{R}^{n}$ consists of fixed points of the time $\tau$-map of
the differential equation%
\[
\dot{x}(t)=A_{0}x(t)+a_{0}+\sum_{i=1}^{m}u_{i}(t)[A_{i}x(t)+a_{i}].
\]
Hence $Y$ consists of a continuum of $\tau$-periodic solutions. Fixing any
$\varepsilon\in\mathcal{P}(\mathbb{R}^{n})$ and $T>0$, one finds that the
unbounded set $Y$ is contained in the strong chain control set $E^{\ast}$.
\end{proof}

Recall that a set in $\mathbb{R}^{n}$ is unbounded if and only if the closure
of its image in $\mathbb{P}^{n}$ has a nonvoid intersection with the
projective equator $\mathbb{P}^{n,0}$ or, equivalently, if the closure of its
image in $\overline{\mathbb{S}^{n,\pm}}$ has nonvoid intersection with the
equator $\mathbb{S}^{n,0}$.

\section{Chain control sets on the sphere for bilinear systems\label{Section5}%
}

We discuss the relation between chain control sets on projective space
$\mathbb{P}^{n}$ and on the sphere $\mathbb{S}^{n}$ for general bilinear
control systems; cf. Bacciotti and Vivalda \cite{BacV13} for a discussion of
complete controllability on these spaces. Then we apply these results to the
system on the Poincar\'{e} sphere induced by an affine control system. In the
rest of this paper, $\varepsilon>0$ denotes constants.

Consider a bilinear control system on $\mathbb{R}^{n+1}$ of the form%
\begin{equation}
\dot{x}(t)=\left[  A_{0}+\sum_{i=1}^{m}u_{i}(t)A_{i}\right]  x(t),\,u\in
\mathcal{U}. \label{bilinear}%
\end{equation}
The trajectories are denoted by $\varphi(t,x,u),t\in\mathbb{R}$. Selgrade's
theorem (cf. Colonius and Kliemann \cite[Theorem 7.1.2]{ColK00}) can be
applied to the corresponding linear control flow $\Phi:\mathbb{R}%
\times\mathcal{U}\times\mathbb{R}^{n+1}\rightarrow\mathcal{U}\times
\mathbb{R}^{n+1}$ and yields the Whitney sum of $\ell$ subbundles%
\[
\mathcal{U}\times\mathbb{R}^{n+1}=\mathcal{V}_{1}\oplus\cdots\oplus
\mathcal{V}_{\ell},
\]
where the projections $\mathbb{P}\mathcal{V}_{i}$ of $\mathcal{V}_{i}$ to the
projective bundle $\mathcal{U}\times\mathbb{P}^{n}$ are the chain recurrent
components of the projectivized flow $\mathbb{P}\Phi$. They yield the chain
control sets $_{\mathbb{P}}E_{1},\ldots,\,_{\mathbb{P}}E_{\ell}$ in
$\mathbb{P}^{n}$ satisfying%
\begin{align}
_{\mathbb{P}}E_{i}  &  =\{\mathbb{P}x\in\mathbb{P}^{n}\left\vert \exists
u\in\mathcal{U}:(u,\mathbb{P}x)\in\mathbb{P}\mathcal{V}_{i}\right.
\},\label{5.2}\\
\mathcal{V}_{i}  &  =\left\{  (u,x)\in\mathcal{U}\times\mathbb{R}%
^{n+1}\left\vert \mathbb{P}\varphi(t,x,u)\in\,_{\mathbb{P}}E_{i}\text{ for all
}t\in\mathbb{R}\right.  \right\}  . \label{5.3}%
\end{align}
We will analyze the chain control sets in $\mathbb{S}^{n}$. For $\varepsilon
,T>0$ let $\zeta$ be a controlled $(\varepsilon,T)$-chain from $s\in
\mathbb{S}^{n}$ to $s^{\prime}\in\mathbb{S}^{n}$ given by $T_{0}%
,\ldots,T_{k-1}\allowbreak\geq T,u^{0},\ldots,u^{k-1},s^{0}=s,s^{1}%
,\ldots,s^{k}=s^{\prime}$ with%
\[
d(\mathbb{S}\varphi(T_{i},s^{i},u^{i}),s^{i+1})<\varepsilon\text{ for all }i.
\]
Since $\mathbb{S}\varphi(T_{i},-s^{i},u^{i})=-\mathbb{S}\varphi(T_{i}%
,s^{i},u^{i})$ a controlled $(\varepsilon,T)$-chain from $-s$ to $-s^{\prime}$
is given by $T_{0},\ldots,T_{k-1}\geq T,-s^{0}=-s,-s^{1},\ldots,-s^{k}%
=-s^{\prime}$ with%
\[
d(\mathbb{S}\varphi(T_{i},-s^{i},u^{i}),-s^{i+1})<\varepsilon.
\]
We denote this chain by $-\zeta$. Recall the definition of the chain reachable
set in Definition \ref{intro2:defchains}.

\begin{lemma}
\label{Lemma_plus_minusS}Let $s,s^{\prime}\in\mathbb{S}^{n}$ with
$\mathbb{P}s,\mathbb{P}s^{\prime}$ in a chain control set $_{\mathbb{P}%
}E\subset\mathbb{P}^{n}$. Then for all $\varepsilon,T>0$ there is a controlled
$(\varepsilon,T)$-chain in $\mathbb{S}^{n}$ from $s$ to $s^{\prime}$ or to
$-s^{\prime}$. Thus the chain reachable set $\mathbf{R}^{c}(s)$ contains
$s^{\prime}$ or $-s^{\prime}$.
\end{lemma}

\begin{proof}
Let $s,s^{\prime}\in\mathbb{S}^{n}$ with $\mathbb{P}s,\mathbb{P}s^{\prime}%
\in\,_{\mathbb{P}}E$. For $\varepsilon,T>0$ consider a controlled
$(\varepsilon,T)$-chain from $\mathbb{P}s$ to $\mathbb{P}s^{\prime}$ given by
$T_{0},\ldots,T_{k-1}\geq T,p^{0}=\mathbb{P}s,p^{1},\ldots,p^{k}%
=\mathbb{P}s^{\prime}$ in $_{\mathbb{P}}E$ and $u^{0},\ldots,u^{k-1}$ with%
\[
d(\mathbb{P}\varphi(T_{i},p^{i},u^{i}),p^{i+1})<\varepsilon\text{ for all }i.
\]
For $s^{i}\in\mathbb{S}^{n}$ with $\mathbb{P}s^{i}=p^{i}$ the definition of
the metric in $\mathbb{P}^{n}$ shows that%
\[
d(\mathbb{P}\varphi(T_{i},p^{i},u^{i}),p^{i+1})=\min\{d(\mathbb{S}%
\varphi(T_{i},s^{i},u^{i}),s^{i+1}),d(\mathbb{S}\varphi(T_{i},s^{i}%
,u^{i}),-s^{i+1})\}.
\]
If $\mathbb{S}\varphi(T_{0},s,u^{0})$ is close to the equator $\mathbb{S}%
^{n,0}$ it may happen that%
\[
d(\mathbb{P}\varphi(T_{0},p^{0},u^{0}),p^{1})=d(\mathbb{S}\varphi(T_{0}%
,s^{0},u^{0}),-s^{1}).
\]
Thus one can continue the controlled $(\varepsilon,T)$-chain in $\mathbb{S}%
^{n}$ with $-s^{1}$ and%
\[
d(\mathbb{S}\varphi(T_{0},s,u^{0}),-s^{1})<\varepsilon.
\]
For the next step note that%
\[
d(\mathbb{P}\varphi(T_{1},p^{1},u^{1}),p^{2})=\min\{d(\mathbb{S}\varphi
(T_{1},-s^{1},u^{1}),s^{2}),d(\mathbb{S}\varphi(T_{1},-s^{1},u^{1}%
),-s^{2})\}.
\]
If $d(\mathbb{P}\varphi(T_{1},p^{1},u^{1}),p^{2})=d(\mathbb{S}\varphi
(T_{1},-s^{1},u^{1}),-s^{2})$, we can continue the chain with $-s^{2}%
\in\mathbb{S}^{n}$ and%
\[
d(\mathbb{S}\varphi(T_{1},-s^{1},u^{1}),-s^{2})<\varepsilon.
\]
If $\mathbb{S}\varphi(T_{1},-s^{1},u^{1})$ is close to the equator it may
happen that%
\[
d(\mathbb{P}\varphi(T_{1},p^{1},u^{1}),p^{2})=d(\mathbb{S}\varphi(T_{1}%
,-s^{1},u^{1}),s^{2}).
\]
Thus we can continue the chain with $s^{2}$ and%
\[
d(\mathbb{S}\varphi(T_{1},-s^{1},u^{1}),s^{2})<\varepsilon.
\]
Continuing in this way we end up in $s^{k}=s^{\prime}$ or in $s^{k}%
=-s^{\prime}$. Thus we have constructed a controlled $(\varepsilon,T)$-chain
in $\mathbb{S}^{n}$ from $s$ to $s^{\prime}$ or to $-s^{\prime}$. Since
$\varepsilon,T>0$ are arbitrary it follows that $s^{\prime}\in\mathbf{R}%
^{c}(s)$ or $-s^{\prime}\in\mathbf{R}^{c}(s)$.
\end{proof}

\begin{lemma}
\label{Lemma_allS}Let $_{\mathbb{P}}E$ be a chain control set in
$\mathbb{P}^{n}$. Then every $s\in\mathbb{S}^{n}$ with $\mathbb{P}%
s\in\,_{\mathbb{P}}E$ is in a chain control set of the system on
$\mathbb{S}^{n}$.
\end{lemma}

\begin{proof}
Let $\varepsilon,T>0$. By Lemma \ref{Lemma_plus_minusS} with $s^{\prime}=s$
there are controlled $(\varepsilon,T)$-chains from $s$ to $s$ or to $-s$. In
the second case, consider the chain $-\zeta$ which starts in $-s$. We end up
in $s$. Concatenation of the chains $\zeta$ from $s$ to $-s$ and $-\zeta$ from
$-s$ to $s$, yields a controlled $(\varepsilon,T)$-chain $\left(
-\zeta\right)  \circ\zeta$ from $s$ to $s$. Since $\varepsilon,T>0$ are
arbitrary, this shows that every $s\in\mathbb{S}^{n}$ with $\mathbb{P}%
s\in\,_{\mathbb{P}}E$ is in a chain control set of the system on
$\mathbb{S}^{n}$.
\end{proof}

The proof of the following lemma is modeled after Bacciotti and Vivalda
\cite[Lemma 3]{BacV13}, where controllability is analyzed; cf. also Colonius,
Santana, Setti \cite[Lemma 3.10]{ColSS22} for the case of control sets.

\begin{lemma}
\label{Lemma_sphere1}(i) Let $_{\mathbb{S}}E$ be a chain control set in
$\mathbb{S}^{n}$. Then the projection of $_{\mathbb{S}}E$ to $\mathbb{P}^{n}$
is contained in a chain control set $_{\mathbb{P}}E$.

(ii) Consider a chain control set $_{\mathbb{P}}E$ in $\mathbb{P}^{n}$.
Suppose that there is $s_{0}\in\mathbb{S}^{n}$ such that $\mathbb{P}s_{0}%
\in\,_{\mathbb{P}}E$ and $-s_{0}\in\mathbf{R}^{c}(s_{0})$. Then there exists a
chain control set $_{\mathbb{S}}E$ in $\mathbb{S}^{n}$ containing
$\{s\in\mathbb{S}^{n}\left\vert \mathbb{P}s\in\,_{\mathbb{P}}E\right.  \}$.
\end{lemma}

\begin{proof}
Assertion (i) is immediate from the definitions. Concerning assertion (ii) let
$s,s^{\prime}\in\mathbb{S}^{n}$ with $\mathbb{P}s,\mathbb{P}s^{\prime}%
\in\,_{\mathbb{P}}E$. We have to show that $s^{\prime}$ is in $\mathbf{R}%
^{c}(s)$. Lemma \ref{Lemma_plus_minusS} shows that $s^{\prime}\in
\mathbf{R}^{c}(s_{0})$ or $s^{\prime}\in\mathbf{R}^{c}(-s_{0})$, and also that
$s_{0}\in\mathbf{R}^{c}(s)$ or $-s_{0}\in\mathbf{R}^{c}(s)$.

If $s^{\prime}\in\mathbf{R}^{c}(s_{0})$ and $s_{0}\in\mathbf{R}^{c}(s)$ it
follows that $s^{\prime}\in\mathbf{R}^{c}(s_{0})\subset\mathbf{R}^{c}(s)$,
and, analogously, the assertion follows if $s^{\prime}\in\mathbf{R}^{c}%
(-s_{0})$ and $-s_{0}\in\mathbf{R}^{c}(s)$. Suppose that $s^{\prime}%
\in\mathbf{R}^{c}(s_{0})$ and $-s_{0}\in\mathbf{R}^{c}(s)$. Using the
assumption on $s_{0}$ one finds that $s^{\prime}\in\mathbf{R}^{c}%
(s_{0})\subset\mathbf{R}^{c}(-s_{0})\subset\mathbf{R}^{c}(s)$. Analogously,
the assertion follows if $s^{\prime}\in\mathbf{R}^{c}(-s_{0})$ and $s_{0}%
\in\mathbf{R}^{c}(s)$.
\end{proof}

The following theorem is the main result of this section. It describes the
chain control sets in $\mathbb{S}^{n}$.

\begin{theorem}
\label{Theorem_twoS}Consider a bilinear control system on $\mathbb{R}^{n+1}$
of the form (\ref{bilinear}) and let $_{\mathbb{P}}E$ be a chain control set
of the induced control system on $\mathbb{P}^{n}$.

(i) The set $_{\mathbb{S}}E_{0}:=\{s\in\mathbb{S}^{n}\left\vert \mathbb{P}%
s\in\,_{\mathbb{P}}E\right.  \}$ is the unique chain control set in
$\mathbb{S}^{n}$ which projects into $_{\mathbb{P}}E$ if and only if there is
$s_{0}\in\mathbb{S}^{n}$ with $\mathbb{P}s_{0}\in\,_{\mathbb{P}}E$ and
$-s_{0}\in\mathbf{R}^{c}(s_{0})$.

(ii) There are two chain control sets $_{\mathbb{S}}E_{1}=-\,_{\mathbb{S}%
}E_{2}$ with%
\begin{equation}
_{\mathbb{S}}E_{1}\cup\,_{\mathbb{S}}E_{2}=\{s\in\mathbb{S}^{n}\left\vert
\mathbb{P}s\in\,_{\mathbb{P}}E\right.  \}, \label{union}%
\end{equation}
if and only if for all $s_{0}\in\mathbb{S}^{n}$ with $\mathbb{P}s_{0}%
\in\,_{\mathbb{P}}E$ it holds that $-s_{0}\not \in \mathbf{R}^{c}(s_{0})$.

(iii) Define for $j\in\{0,1,2\}$ and $u\in\mathcal{U}$%
\[
V_{j}(u):=\left\{  x\in\mathbb{R}^{n+1}\left\vert \mathbb{S}\varphi
(t,x,u)\in\,_{\mathbb{S}}E_{j}\text{ for all }t\in\mathbb{R}\right.  \right\}
.
\]
In case (i) it follows that $V_{0}(u)$ is a linear subspace. In case (ii) it
follows for $j=1,2$ that $V_{j}(u)$ is a convex cone.

(iv) There are $\ell_{1}$ chain control sets on $\mathbb{S}^{n}$ denoted by
$_{\mathbb{S}}E_{1},\ldots,\allowbreak\,_{\mathbb{S}}E_{\ell_{1}}$ with
$1\leq\ell_{1}\leq2\ell\leq2(n+1)$.
\end{theorem}

\begin{proof}
(i) Suppose that there is $\mathbb{P}s_{0}\in\,_{\mathbb{P}}E$ with $-s_{0}%
\in\mathbf{R}^{c}(s_{0})$. By Lemma \ref{Lemma_sphere1}(ii) there is a chain
control set $_{\mathbb{S}}E$ on the unit sphere containing $\{s\in
\mathbb{S}^{n}\left\vert \mathbb{P}s\in\,_{\mathbb{P}}E\right.  \}$, hence the
projection of $_{\mathbb{S}}E$ to projective space contains $_{\mathbb{P}}E$.
Using Lemma \ref{Lemma_sphere1}(i) one concludes that $_{\mathbb{S}}%
E=\{s\in\mathbb{S}^{n}\left\vert \mathbb{P}s\in\,_{\mathbb{P}}E\right.  \}$.

(ii) Suppose that for all $s_{0}\in\mathbb{S}^{n}$ with $\mathbb{P}s_{0}%
\in\,_{\mathbb{P}}E$ it holds that $-s_{0}\not \in \mathbf{R}^{c}(s_{0})$. Fix
a point $s_{0}\in\mathbb{S}^{n}$ with $\mathbb{P}s_{0}\in\,_{\mathbb{P}}E$ and
define%
\begin{align*}
A^{+}  &  :=\left\{  s\in\mathbb{S}^{n}\left\vert \,\mathbb{P}s\in
\,_{\mathbb{P}}E\text{ and }s\in\mathbf{R}^{c}(s_{0})\cap\mathbf{C}^{c}%
(s_{0})\right.  \right\}  ,\\
A^{-}  &  :=\left\{  s\in\mathbb{S}^{n}\left\vert \,\mathbb{P}s\in
\,_{\mathbb{P}}E\text{ and }s\in\mathbf{R}^{c}(-s_{0})\cap\mathbf{C}%
^{c}(-s_{0})\right.  \right\}  .
\end{align*}
It follows that $A^{-}=-A^{+}$ and the sets $A^{+}$ and $A^{-}$ are disjoint
since $s\in A^{+}\cap A^{-}$ implies $s\in\mathbf{R}^{c}(s_{0})$ and
$s\in\mathbf{C}^{c}(-s_{0})$ which implies that $-s_{0}\in\mathbf{R}^{c}%
(s_{0})$ contradicting the assumption. By Lemma \ref{Lemma_allS} there is a
chain control set $_{\mathbb{S}}E_{1}$ with $s_{0}\in\,_{\mathbb{S}}E_{1}$.
Lemma \ref{Lemma_sphere1}(i)\textbf{ }implies that $\mathbb{P}(_{\mathbb{S}%
}E_{1})\subset\,_{\mathbb{P}}E$ and Theorem \ref{Theorem_E_omega} implies that
$_{\mathbb{S}}E_{1}=\mathbf{R}^{c}(s_{0})\cap\mathbf{C}^{c}(s_{0})$, hence
$A^{+}=\,_{\mathbb{S}}E_{1}$.

Since by our assumption $s_{0}\not \in \mathbf{R}^{c}(-s_{0})$ it follows by
Lemma \ref{Lemma_plus_minusS} that $-s_{0}\in\mathbf{R}^{c}(-s_{0})$, hence
there is a chain control set $_{\mathbb{S}}E_{2}$ with $-s_{0}\in
\,_{\mathbb{S}}E_{2}$. By Theorem \ref{Theorem_E_omega} it follows that
$_{\mathbb{S}}E_{2}=\mathbf{R}^{c}(-s_{0})\cap\mathbf{C}^{c}(-s_{0})$, hence
$A^{-}=\,_{\mathbb{S}}E_{2}$.

Suppose that $\mathbb{P}s\in\,_{\mathbb{P}}E$. Then Lemma
\ref{Lemma_plus_minusS} implies that $s\in\mathbf{R}^{c}(s_{0})$ or
$-s\in\mathbf{R}^{c}(s_{0})$ and $s\in\mathbf{C}^{c}(s_{0})$ or $-s\in
\mathbf{C}^{c}(s_{0})$. If $s\in\mathbf{R}^{c}(s_{0})$ and $-s\in
\mathbf{C}^{c}(s_{0})$ it follows that $s\in\mathbf{C}^{c}(-s_{0})$ which
leads to the contradiction $-s_{0}\in\mathbf{R}^{c}(s_{0})$. Similarly,
$-s\in\mathbf{R}^{c}(s_{0})$ and $s\in\mathbf{C}^{c}(s_{0})$ leads to the
contradiction $-s_{0}\in\mathbf{R}^{c}(s_{0})$. We conclude that
$\mathbb{P}s\in\,_{\mathbb{P}}E$ implies $s\in\mathbf{R}^{c}(s_{0}%
)\cap\mathbf{C}^{c}(s_{0})$ or $s\in\mathbf{R}^{c}(-s_{0})\cap\mathbf{C}%
^{c}(-s_{0})$. This shows that
\begin{equation}
\{s\in\mathbb{S}^{n}\left\vert \mathbb{P}s\in\,_{\mathbb{P}}E\right.
\}=A^{+}\cup A^{-}=\,_{\mathbb{S}}E_{1}\cup\,_{\mathbb{S}}E_{2}.
\label{contained1}%
\end{equation}
Together with the first part of the proof it follows that the conditions in
(i) and (ii) are also necessary for the existence of a single chain control
set $_{\mathbb{S}}E_{0}$ and of two chain control sets $_{\mathbb{S}}E_{1}$
and $_{\mathbb{S}}E_{2}$, respectively.

(iii) In the situation of case (i), one has%
\[
\left\{  x\in\mathbb{R}^{n+1}\left\vert \mathbb{S}x\in\,_{\mathbb{S}}%
E_{0}\right.  \right\}  =\left\{  x\in\mathbb{R}^{n+1}\left\vert
\mathbb{P}x\in\,_{\mathbb{P}}E\right.  \right\}  ,
\]
and the assertion for $_{\mathbb{S}}E_{0}$ follows since by Selgrade's theorem
the fibers of $\mathcal{V}_{i}$ are linear; cf. formula (\ref{5.3}). In the
situation of the case (ii), the assertion follows if for $j=1,2$ and
$u\in\mathcal{U}$
\[
x_{1},x_{2}\in V_{j}(u)\text{ implies }\alpha x_{1}+\beta x_{2}\in
V_{j}(u)\text{ for all }\alpha,\beta>0.
\]
It suffices to prove that $\mathbb{S}(\alpha x_{1}+\beta x_{2})\in
\,_{\mathbb{S}}E_{j}$ since one may replace this point by $\mathbb{S}%
\varphi(t,\alpha x_{1}+\beta x_{2},u),t\in\mathbb{R}$. Note first that
$\mathbb{P}\left(  \alpha x_{1}+\beta x_{2}\right)  \in\,_{\mathbb{P}}E_{j}$
by Selgrade's theorem. Furthermore, $\alpha x_{1}+\beta x_{2}\not =0$ since
$\alpha x_{1}+\beta x_{2}=0$ implies $\mathbb{S}x_{1}=-\mathbb{S}x_{2}%
\in\,_{\mathbb{S}}E_{1}\cap\,_{\mathbb{S}}E_{2}$ contradicting $_{\mathbb{S}%
}E_{1}\not =\,_{\mathbb{S}}E_{2}$. It follows for $t\in\mathbb{R}$ that
$\mathbb{S}\varphi(t,\alpha x_{1}+\beta x_{2},u)\not =0$. Since%
\[
\mathbb{S}(\alpha x_{1}+\beta x_{2})=\frac{\alpha x_{1}+\beta x_{2}%
}{\left\Vert \alpha x_{1}+\beta x_{2}\right\Vert }=\frac{x_{1}+\frac{\beta
}{\alpha}x_{2}}{\left\Vert x_{1}+\frac{\beta}{\alpha}x_{2}\right\Vert },
\]
it suffices to show that for $s_{1},s_{2}\in\mathbb{S}^{n}$ with $s_{1}%
,s_{2}\in\,_{\mathbb{S}}E_{j}$ and $\alpha\in(0,1)$%
\[
\mathbb{S}(s_{1}+\alpha s_{2})=\frac{s_{1}+\alpha s_{2}}{\left\Vert
s_{1}+\alpha s_{2}\right\Vert }\in\,_{\mathbb{S}}E_{j}.
\]
Note that $\mathbb{P}(s_{1}+\alpha s_{2})\in\,_{\mathbb{P}}E_{j}$. The
Hausdorff distance between the compact sets $_{\mathbb{S}}E_{1}$ and
$_{\mathbb{S}}E_{2}$ is positive, hence (\ref{union}) implies $\mathbb{S}%
(s_{1}+\alpha s_{2})\in\,_{\mathbb{S}}E_{j}$ for small $\alpha>0$. Let%
\[
\alpha^{0}:=\sup\left\{  \alpha\in(0,1)\left\vert \mathbb{S}(s_{1}%
+\alpha^{\prime}s_{2})\in\mathbb{\,}_{\mathbb{S}}E_{j}\text{ for all }%
\alpha^{\prime}\in(0,\alpha)\right.  \right\}  .
\]
If $\alpha^{0}<1$ there exists $\alpha\in(\alpha^{0},1)$ with $\mathbb{S}%
(s_{1}+\alpha s_{2})\not \in \,_{\mathbb{S}}E_{1}\cup\,_{\mathbb{S}}E_{2}$.
This contradicts (\ref{union}). Hence it follows that $\mathbb{S}(s_{1}+\alpha
s_{2})\in\,_{\mathbb{S}}E_{j}$ for all $\alpha\in(0,1)$.

(iv) By Selgrade's theorem the number $\ell$ of chain control sets in
$\mathbb{P}^{n}$ satisfies $1\leq\ell\leq n+1$. Hence assertions (i) and (ii)
imply (iv).
\end{proof}

Next, we describe the relation of the chain control sets on the sphere to the
Selgrade bundles $\mathcal{V}_{i}$; cf. (\ref{5.3}).

\begin{corollary}
\label{Corollary_twoS}For a bilinear system of the form (\ref{bilinear}) the
projection to $\mathbb{S}^{n}$ of the Selgrade bundle $\mathcal{V}_{i}$ given
by
\[
\pi_{\mathbb{S}}\mathcal{V}_{i}:=\left\{  s\in\mathbb{S}^{n}\left\vert
\exists(u,x)\in\mathcal{V}_{i}\text{ with }s=\frac{x}{\left\Vert x\right\Vert
}\right.  \right\}  ,
\]
equals $_{\mathbb{S}}E_{0}$ or $_{\mathbb{S}}E_{1}\cup\,_{\mathbb{S}}E_{2}$,
where $_{\mathbb{S}}E_{j}$ are the chain control sets from Theorem
\ref{Theorem_twoS}(i) or (ii), respectively.
\end{corollary}

\begin{proof}
The set $\pi_{\mathbb{S}}\mathcal{V}_{i}$ projects onto $_{\mathbb{P}}%
E_{i}\subset\mathbb{P}^{n}$, hence $\pi_{\mathbb{S}}\mathcal{V}_{c}%
^{1}=\left\{  s\in\mathbb{S}^{n}\left\vert \mathbb{P}s\in\,_{\mathbb{P}}%
E_{i}\right.  \right\}  $ and the assertion follows from Theorem
\ref{Theorem_twoS}.
\end{proof}

In the rest of this section, we consider the control system induced by an
affine control system on the Poincar\'{e} sphere.

\begin{corollary}
\label{Corollary_three}Consider an affine control system of the form
(\ref{affine1}) and the induced control systems on $\mathbb{P}^{n}$ and
$\mathbb{S}^{n}$. Let $_{\mathbb{P}}E_{c}^{1}$ be the central chain control
set on $\mathbb{P}^{n}$. Then the system on $\mathbb{S}^{n}$ satisfies
\[
\left\{  s\in\mathbb{S}^{n}\left\vert \mathbb{P}s\in E_{c}^{1}\right.
\right\}  =\,_{\mathbb{S}}E_{c,0}\text{ or }\left\{  s\in\mathbb{S}%
^{n}\left\vert \mathbb{P}s\in E_{c}^{1}\right.  \right\}  =\,_{\mathbb{S}%
}E_{c,1}\cup\,_{\mathbb{S}}E_{c,2},
\]
where $_{\mathbb{S}}E_{c,0}$ is a chain control set with $_{\mathbb{S}}%
E_{c,0}=-\,_{\mathbb{S}}E_{c,0}$ and $_{\mathbb{S}}E_{c,2}=-\,_{\mathbb{S}%
}E_{c,1}$ are two chain control sets, respectively.
\end{corollary}

\begin{proof}
This is a consequence of Theorem \ref{Theorem_twoS}(i),(ii).
\end{proof}

The chain control sets $_{\mathbb{S}}E_{c,j},j\in\{0,1,2\}$, are the only
chain control sets on $\mathbb{S}^{n}$, which are not contained in the equator
$\mathbb{S}^{n,0}$. We call them the central chain control sets on
$\mathbb{S}^{n}$.

The following proposition shows that the intersection of a chain control set
$_{\mathbb{S}}E_{c,j}$ in $\mathbb{S}^{n}$ with the equator $\mathbb{S}^{n,0}$
contains a chain control set in $\mathbb{S}^{n,0}$. This result and its proof
are similar to Colonius, Santana, and Setti \cite[Lemma 6.8]{ColSS23}, where
an analogous result is given on projective spaces. Theorem \ref{Theorem_twoS}
also applies to the bilinear control system on $\mathbb{R}^{n}$ given by the
homogeneous part (\ref{bilinear0}). For the induced system on the sphere
$\mathbb{S}^{n-1}$ this yields chain control sets $_{\mathbb{S}}E^{\hom}$ in
$\mathbb{S}^{n-1}$. We identify $\mathbb{S}^{n-1}\times\{0\}$ with $\left\{
s\in\mathbb{S}^{n}\left\vert s_{n+1}=0\right.  \right\}  $, hence the chain
control sets contained in the equator $\mathbb{S}^{n,0}$ have the form
$_{\mathbb{S}}E^{\hom}\times\{0\}$.

\begin{proposition}
\label{Proposition5.7}Consider an affine control system of the form
(\ref{affine1}) on $\mathbb{R}^{n}$, and let $_{\mathbb{S}}E_{c,j}%
,j\in\{0,1,2\}$, be the central chain control sets in the Poincar\'{e} sphere
$\mathbb{S}^{n}$.

(i) If $_{\mathbb{S}}E_{c,j}\cap\left(  _{\mathbb{S}}E^{\hom}\times
\{0\}\right)  \not =\varnothing$ for a chain control set $_{\mathbb{S}}%
E^{\hom}\subset\mathbb{S}^{n-1}$ of the projectivized homogeneous part, then
$_{\mathbb{S}}E^{\hom}\times\{0\}\subset\,_{\mathbb{S}}E_{c,j}$.

(ii) If $_{\mathbb{S}}E_{c,j}\cap\mathbb{S}^{n,0}\not =\varnothing$ it
contains $_{\mathbb{S}}E^{\hom}\times\{0\}$ for a chain control set
$_{\mathbb{S}}E^{\hom}$ in $\mathbb{S}^{n-1}$ of the homogeneous part
(\ref{bilinear0}).

(iii) The chain control sets $_{\mathbb{S}}E_{c,j}$ and $_{\mathbb{S}}E^{\hom
}$ are also strong chain control sets.
\end{proposition}

\begin{proof}
(i) We will show that the set $_{\mathbb{S}}E_{c,j}\cup\left(  _{\mathbb{S}%
}E^{\hom}\times\{0\}\right)  $ is chain controllable. Then the maximality
property of the chain control set $_{\mathbb{S}}E_{c,j}$ implies that
$_{\mathbb{S}}E^{\hom}\times\{0\}\subset\,_{\mathbb{S}}E_{c,j}$.

Consider $x\in\,_{\mathbb{S}}E_{c,j}$ and $y\in\,_{\mathbb{S}}E^{\hom}%
\times\{0\}$ and let $\varepsilon,T>0$. Fix $z\in\,_{\mathbb{S}}E_{c,j}%
\cap\left(  _{\mathbb{S}}E^{\hom}\times\{0\}\right)  $. There are controlled
$(\varepsilon,T)$-chains $\zeta_{1}$ and $\zeta_{2}$ from $x$ to $z$ and from
$z$ to $x$, resp. For the system restricted to $\mathbb{S}^{n,0}$, there exist
controlled $(\varepsilon,T)$-chains $\zeta_{3}$ and $\zeta_{4}$ from $z$ to
$y$ and from $y$ to $z$, respectively. Then the concatenations $\zeta_{3}%
\circ\zeta_{1}$ and $\zeta_{2}\circ\zeta_{4}$ are controlled $(\varepsilon
,T)$-chains from $x$ to $y$ and from $y$ to $x$, resp. This concludes the
proof of assertion (i) since $\varepsilon,T>0$ are arbitrary.

(ii) Let $(x,0)\in\,_{\mathbb{S}}E_{c,j}\cap\mathbb{S}^{n,0}$. Then there
exists a control $u\in\mathcal{U}$ with $\mathbb{S}\psi^{1}%
(t,x,0,u)\allowbreak\in\,_{\mathbb{S}}E_{c,j}\cap\mathbb{S}^{n,0}%
,t\in\mathbb{R}$, by Theorem \ref{Theorem_E_omega} and invariance of
$\mathbb{S}^{n,0}$. Since $_{\mathbb{S}}E_{c,j}\cap\mathbb{S}^{n,0}$ is
compact, it follows that the $\omega$-limit set%
\[
\left\{  \left.  y=\lim_{k\rightarrow\infty}\mathbb{S}\psi^{1}(t_{k}%
,x,0,u)\right\vert t_{k}\rightarrow\infty\right\}  \subset\,_{\mathbb{S}%
}E_{c,j}\cap\mathbb{S}^{n,0}%
\]
is nonvoid. Hence Colonius and Kliemann \cite[Corollary 4.3.12]{ColK00}
implies that there exists a chain control set of the system restricted to
$\mathbb{S}^{n,0}$ containing the $\omega$-limit set. Thus there is a chain
control set $_{\mathbb{S}}E^{\hom}$ in $\mathbb{S}^{n-1}$ of the homogeneous
part with $_{\mathbb{S}}E_{c,j}\cap\left(  _{\mathbb{S}}E^{\hom}%
\times\{0\}\right)  \not =\varnothing$. Now the assertion follows from (i).

(iii) Since $\mathbb{S}^{n}$ and $\mathbb{S}^{n-1}$ are compact, Proposition
\ref{Proposition_compact}(iv) implies that $_{\mathbb{S}}E$ and $_{\mathbb{S}%
}E^{\hom}$ are also strong chain control sets.
\end{proof}

The following theorem describes the intersection of a chain control set on the
Poincar\'{e} sphere $\mathbb{S}^{n}$ with the equator $\mathbb{S}^{n,0}$.

\begin{theorem}
Consider an affine control system of the form (\ref{affine1}) on
$\mathbb{R}^{n}$, and suppose that the central chain control set
$_{\mathbb{P}}E_{c}$ on the projective Poincar\'{e} sphere $\mathbb{P}^{n}$
has a nonvoid intersection $_{\mathbb{P}}E_{c}\cap\left(  _{\mathbb{P}}%
E^{\hom}\times\{0\}\right)  \,$ with a chain control set $_{\mathbb{P}}%
E^{\hom}$ in $\mathbb{P}^{n-1}$for the homogeneous part (\ref{bilinear0}).
Denote by $_{\mathbb{S}}E_{j}^{\hom}\subset\mathbb{S}^{n-1},j=0$ or $j=1,2$,
the chain control sets mapped onto $_{\mathbb{P}}E^{\hom}$.

(i) Consider the case where the central chain control set $_{\mathbb{S}%
}E_{c,0}\ $on $\mathbb{S}^{n}$ exists. Then for the chain control sets
$_{\mathbb{S}}E_{j}^{\hom},j=0$ or $j=1,2$, one has $_{\mathbb{S}}E_{j}^{\hom
}\times\{0\}\subset\,_{\mathbb{S}}E_{c,0}\cap\mathbb{S}^{n,0}$.

(ii) Consider the case where the central chain control sets $_{\mathbb{S}%
}E_{c,1}$ and $_{\mathbb{S}}E_{c,2}\ $exist. Then the chain control sets
$_{\mathbb{S}}E_{1}^{\hom}$ and $\,_{\mathbb{S}}E_{2}^{\hom}$ exist, and for
$j=1,2$ exactly one of the sets $_{\mathbb{S}}E_{1}^{\hom}\times\{0\}$ and
$_{\mathbb{S}}E_{2}^{\hom}\times\{0\}$ is contained in $_{\mathbb{S}}%
E_{c,j}\cap\mathbb{S}^{n,0}$. The case where $_{\mathbb{S}}E_{0}^{\hom}$
exists cannot occur.
\end{theorem}

\begin{proof}
(i) The chain control set $_{\mathbb{S}}E_{c,0}\ $is the preimage of
$_{\mathbb{P}}E_{c}\ $and either $_{\mathbb{S}}E_{j}^{\hom}$ or the union
$_{\mathbb{S}}E_{1}^{\hom}\cup\,_{\mathbb{S}}E_{2}^{\hom}$ forms the preimage
of $_{\mathbb{P}}E^{\hom}$. Since $\mathbb{S}^{n,0}$ is the preimage of
$\mathbb{P}^{n,0}$ it follows that $_{\mathbb{S}}E_{j}^{\hom}\times
\{0\}\subset\,_{\mathbb{S}}E_{c,0}\cap\mathbb{S}^{n,0}$ for $j=0,1,2$.

(ii) The same argument shows that $_{\mathbb{S}}E_{c,j}\cap\mathbb{S}%
^{n,0}\not =\varnothing,j=1,2$. Suppose that $_{\mathbb{S}}E_{1}^{\hom
}=-\,_{\mathbb{S}}E_{2}^{\hom}$ exist. Then it holds that
\[
\left(  _{\mathbb{S}}E_{1}^{\hom}\times\{0\}\right)  \cup\left(
-\,_{\mathbb{S}}E_{2}^{\hom}\times\{0\}\right)  =\,_{\mathbb{S}}E_{c,1}%
\cup\,_{\mathbb{S}}E_{c,2}.
\]
If $_{\mathbb{S}}E_{1}^{\hom}\cap\,_{\mathbb{S}}E_{c,1}\not =\varnothing$ it
follows by Proposition \ref{Proposition5.7}(i) that $_{\mathbb{S}}E_{1}^{\hom
}\subset\,_{\mathbb{S}}E_{c,1}$. Then the set $_{\mathbb{S}}E_{2}^{\hom}%
\times\{0\}=-\left(  _{\mathbb{S}}E_{1}^{\hom}\times\{0\}\right)  $ cannot
intersect $_{\mathbb{S}}E_{c,1}$ since this would lead to an element $s_{0}$
with $\mathbb{P}s_{0}\in\,_{\mathbb{P}}E_{c,1}$ and $-s_{0}\in\,_{\mathbb{S}%
}E_{c,1}\subset\mathbf{R}^{c}(s_{0})$ contradicting Theorem \ref{Theorem_twoS}%
(ii). The same argument applies to $_{\mathbb{S}}E_{c,2}$. If $_{\mathbb{S}%
}E_{0}^{\hom}$ exists, then the chain control set $_{\mathbb{S}}E_{0}^{\hom
}\times\{0\}$ would be contained in the intersection $_{\mathbb{S}}E_{c,1}%
\cap\,_{\mathbb{S}}E_{c,2}$, which is a contradiction.
\end{proof}

\section{Chain control sets on hemispheres\label{Section6}}

By Corollary \ref{Corollary_relation}\ the strong chain control sets on the
northern hemisphere $\mathbb{S}^{n,+}$ of the Poincar\'{e} sphere correspond
to the strong chain control sets $E^{\ast}$ of the affine system
(\ref{affine1}) and are given by $_{\mathbb{S}}E^{\ast}=h(E^{\ast})$. By
Proposition \ref{Proposition_compact}(ii) every strong chain control set is
contained in a chain control set. This implies that the strong chain control
sets $_{\mathbb{S}}E^{\ast}$ are contained in chain control sets in
$\mathbb{S}^{n,+}$ and in chain control sets of the system on the closure
$\overline{\mathbb{S}^{n,+}}$. These chain control sets are contained in one
of the central chain control set $_{\mathbb{S}}E_{c,j},j\in\{0,1,2\}$ on
$\mathbb{S}^{n}$. In this section, we will further analyze the chain control
sets in $\overline{\mathbb{S}^{n,+}}$.

Theorem \ref{Theorem_Cor32}(ii) determines the intersection of the central
Selgrade bundle $\mathcal{V}_{c}^{1}$ with $\mathcal{U}\times\mathbb{S}^{n,0}%
$. In this section, we restrict the analysis to the case $\ell^{0}=1$, hence
there is a unique subbundle $\mathcal{V}_{i}^{\infty},i=\ell^{+}+1$, contained
in $\mathcal{V}_{c}^{1}$ and%
\[
\mathcal{V}_{c}^{1}\cap\left(  \mathcal{U}\times\mathbb{S}^{n,0}\right)
=\mathcal{V}_{\ell^{+}+1}^{\infty}\cap\left(  \mathcal{U}\times\mathbb{S}%
^{n,0}\right)  \text{ with }\dim\mathcal{V}_{c}^{1}=1+\dim\mathcal{V}%
_{\ell^{+}+1}^{\infty}.
\]
Here $\mathcal{V}_{\ell^{+}+1}^{\infty}=\mathcal{V}_{\ell^{+}+1}%
\times\{0\}\subset\mathcal{V}_{c}^{1}$, where $\mathcal{V}_{\ell^{+}+1}$ is a
Selgrade bundle of the homogeneous part (\ref{bilinear0}). The Selgrade
bundles $\mathcal{V}_{i}$ are related to the chain control sets $_{\mathbb{P}%
}E_{i}^{\hom}$ in $\mathbb{P}^{n-1}$ of the homogeneous part (\ref{bilinear0})
by%
\[
\mathcal{V}_{i}=\left\{  (u,x)\in\mathcal{U}\times\mathbb{R}^{n}\left\vert
\mathbb{P}\varphi(t,x,u)\in\,_{\mathbb{P}}E_{i}^{\hom}\text{ for all }%
t\in\mathbb{R}\right.  \right\}  .
\]
We will analyze chain controllability properties in the closure $\overline
{\mathbb{S}^{n,+}}=\mathbb{S}^{n,+}\cup\mathbb{S}^{n,0}$ of the northern
hemisphere. On the level of chain control sets this means that for the central
chain control set $_{\mathbb{P}}E_{c}$ in $\mathbb{P}^{n}$ there is a single
chain control set $_{\mathbb{P}}E^{\hom}$ of the projectivized homogeneous
part with%
\[
_{\mathbb{P}}E_{c}\cap\mathbb{P}^{n,0}=\,_{\mathbb{P}}E^{\hom}\times\{0\}.
\]
Actually, we even require that $_{\mathbb{P}}E^{\hom}$ is the projection of a
single chain control set $_{\mathbb{S}}E_{0}^{\hom}$ on $\mathbb{S}^{n-1}$.
With this assumption on the level of chain control sets on $\mathbb{S}^{n}$ we
prove the following lemma.

\begin{lemma}
\label{Lemma6.1}Consider an affine control system of the form (\ref{affine1})
on $\mathbb{R}^{n}$ and assume that a central chain control set $_{\mathbb{S}%
}E_{c,j},j\in\{0,1,2\}$, in $\mathbb{S}^{n}$ satisfies%
\begin{equation}
_{\mathbb{S}}E_{c,j}\cap\mathbb{S}^{n,0}=\,_{\mathbb{S}}E_{0}^{\hom}%
\times\{0\}, \label{6.1}%
\end{equation}
where the chain control set $_{\mathbb{S}}E_{0}^{\hom}$ is the preimage in
$\mathbb{S}^{n-1}$ of a chain control set $_{\mathbb{P}}E^{\hom}$ of the
projectivized homogeneous part (\ref{bilinear0}). Then the intersection
$_{\mathbb{S}}E_{c,j}\cap\overline{\mathbb{S}^{n,+}}$ is chain controllable on
$\overline{\mathbb{S}^{n,+}}$.
\end{lemma}

\begin{proof}
Let $s,s^{\prime}\in\,_{\mathbb{S}}E_{c,j}\cap\mathbb{S}^{n,+}$. We construct
for $\varepsilon,T>0$ a controlled $(\varepsilon,T)$-chain in $\overline
{\mathbb{S}^{n,+}}$ from $s$ to $s^{\prime}$. Choose a controlled
$(\varepsilon/2,T)$-chain $\zeta$ in $\mathbb{S}^{n}$ from $s\in
\,_{\mathbb{S}}E_{c,j}\cap\mathbb{S}^{n,+}$ to $s^{\prime}\in\,_{\mathbb{S}%
}E_{c,j}\cap\mathbb{S}^{n,+}$ given by $T_{0},\ldots,T_{k-1}\geq
T,u^{0},\ldots,u^{k-1}\in\mathcal{U},s^{0}=s,s^{1},\ldots,s^{k}=s^{\prime}$ in
$_{\mathbb{S}}E_{c,j}$ with%
\[
d(\mathbb{S}\varphi(T_{i},s^{i},u^{i}),s^{i+1})<\varepsilon/2\text{ for all
}i.
\]
Since $\mathbb{S}^{n}$ is compact, the chain may be chosen in $_{\mathbb{S}%
}E_{c,j}$. If all $s^{i}$ are in $\mathbb{S}^{n,+}$, we are done. Otherwise
there is a minimal $i$, called a crossing index, with%
\begin{equation}
s^{i}\in S^{n,+}\text{ and }s^{i+1}\in\mathbb{S}^{n,-}\cup\mathbb{S}^{n,0}.
\label{crossing}%
\end{equation}
\textbf{Step 1: }If $s^{i+1}\in\mathbb{S}^{n,0}$, define $\widetilde
{s}:=s^{i+1}$. Otherwise $s^{i+1}\in\mathbb{S}^{n,-}$. By invariance of
$\mathbb{S}^{n,+}$ we know that $\mathbb{S}\varphi(T_{i},s^{i},u^{i})\in
S^{n,+}$ and the estimate%
\[
d(\mathbb{S}\varphi(T_{i},s^{i},u^{i}),s^{i+1})<\varepsilon/2
\]
holds. By Theorem \ref{Theorem_twoS}(iii) it follows that%
\[
\left\{  \left.  \frac{\alpha\mathbb{S}\varphi(T_{i},s^{i},u^{i}%
)+(1-\alpha)s^{i+1}}{\left\Vert \alpha\mathbb{S}\varphi(T_{i},s^{i}%
,u^{i})+(1-\alpha)s^{i+1}\right\Vert }\right\vert \alpha\in\lbrack
0,1]\right\}  \subset\,_{\mathbb{S}}E_{c,j}.
\]
The last component of $\mathbb{S}\varphi(T_{i},s^{i},u^{i})$ is positive, and
the last component of $s^{i+1}$ is negative. Hence there is $\alpha\in(0,1)$
such that%
\[
\widetilde{s}:=\frac{\alpha\mathbb{S}\varphi(T_{i},s^{i},u^{i})+(1-\alpha
)s^{i+1}}{\left\Vert \alpha\mathbb{S}\varphi(T_{i},s^{i},u^{i})+(1-\alpha
)s^{i+1}\right\Vert }\in\,_{\mathbb{S}}E_{c,j}\cap\mathbb{S}^{n,0}.
\]
This implies%
\[
d(\mathbb{S}\varphi(T_{i},s^{i},u^{i}),\widetilde{s})<\varepsilon/2\text{ and
}d(\widetilde{s},s^{i+1})<\varepsilon/2.
\]
By assumption (\ref{6.1}) it follows that there is $\widehat{s}\in
\,_{\mathbb{S}}E_{0}^{\hom}$ with $\widetilde{s}=(\widehat{s},0)$. Since
$_{\mathbb{S}}E_{0}^{\hom}=-\,_{\mathbb{S}}E_{0}^{\hom}$ there is a controlled
$(\varepsilon/2,T)$-chain of the homogeneous part from $\widehat{s}$ to
$-\widehat{s}$. This defines a controlled $(\varepsilon/2,T)$-chain $\zeta
^{1}$ in $\mathbb{S}^{n}$ from $(\widehat{s},0)$ to $(-\widehat{s},0)$ and%
\[
d(-\widetilde{s},-s^{i+1})=d(-(\widehat{s},0),-s^{i+1})=d((\widehat
{s},0),s^{i+1})<\varepsilon/2.
\]
Replace the final point $(-\widehat{s},0)$ of $\zeta^{1}$ by $-s^{i+1}$, hence
the jump length is bounded by%
\[
d(\mathbb{S}\varphi(T_{i},s^{i},u^{i}),-s^{i+1})\leq d(\mathbb{S}\varphi
(T_{i},s^{i},u^{i}),-\widetilde{s})+d(-\widetilde{s},-s^{i+1})<\varepsilon
/2+\varepsilon/2=\varepsilon.
\]
Then we insert this controlled $(\varepsilon,T)$-chain between $\mathbb{S}%
\varphi(T_{i},s^{i},u^{i})$ and $-s^{i+1}$. The concatenation is a controlled
$(\varepsilon,T)$-chain from $s$ to $-s^{i+1}$.

\textbf{Step 2: }Next we consider what happens for $i+2$.

(a) Suppose first that $s^{i+2}\in S^{n,+}$. The construction in \textbf{Step
1} has led us to the point $-s^{i+1}\in\mathbb{S}^{n,+}$. We know that%
\[
d(\mathbb{S}\varphi(T_{i+1},s^{i+1},u^{i+1}),s^{i+2})<\varepsilon.
\]
By invariance of $\mathbb{S}^{n,-}$ we find, similarly as in \textbf{Step 1} a
point $(\widehat{s},0)$ with $\widehat{s}\in\,_{\mathbb{S}}E^{\hom}$ with%
\[
d(\mathbb{S}\varphi(T_{i+1},s^{i+1},u^{i+1}),(\widehat{s},0))<\varepsilon
/2\text{ and }d((\widehat{s},0),s^{i+2})<\varepsilon/2.
\]
This implies%
\[
d(\mathbb{S}\varphi(T_{i+1},-s^{i+1},u^{i+1}),(-\widehat{s},0))=d(\mathbb{S}%
\varphi(T_{i+1},s^{i+1},u^{i+1}),(\widehat{s},0))<\varepsilon/2.
\]
We insert a controlled $(\varepsilon/2,T)$-chain $\zeta^{2}$ from
$(-\widehat{s},0)$ to $(\widehat{s},0)$ and replace the final point by
$s^{i+2}$. The concatenation is a controlled $(\varepsilon,T)$-chain from $s$
to $s^{i+2}\in\mathbb{S}^{n,+}$.

(b) Suppose that $s^{i+2}\in S^{n,-}\cup\mathbb{S}^{n,0}$. Since%
\begin{equation}
d(\mathbb{S}\varphi(T_{i+1},-s^{i+1},u^{i+1}),-s^{i+2})=d(\mathbb{S}%
\varphi(T_{i+1},s^{i+1},u^{i+1}),s^{i+2})<\varepsilon\label{final}%
\end{equation}
we have a controlled $(\varepsilon,T)$-chain from $s$ to $-s^{i+2}%
\in\mathbb{S}^{n,+}\cup\mathbb{S}^{n,0}$. If $s^{i+3}\in\mathbb{S}^{n,-}%
\cup\mathbb{S}^{n,0}$ we obtain a controlled $(\varepsilon,T)$-chain from $s$
to $-s^{i+3}\in\mathbb{S}^{n,+}\cup\mathbb{S}^{n,0}$. If $s^{i+3}\in
\mathbb{S}^{n,+}$ we have%
\[
d(\mathbb{S}\varphi(T_{i+2},-s^{i+2},u^{i+2}),-s^{i+3})<\varepsilon.
\]
The construction in \textbf{Step 1} yields a point $\widetilde{s}=(\widehat
{s},0)$ with $\widehat{s}\in E_{0}^{\hom}$ with%
\[
d(\mathbb{S}\varphi(T_{i+2},-s^{i+2},u^{i+2}),\widetilde{s})<\varepsilon
/2\text{ and }d(\widetilde{s},-s^{i+3})<\varepsilon/2.
\]
Then one obtains a controlled $(\varepsilon,T)$- chain from $-s^{i+2}$ to
$s^{i+3}\in\mathbb{S}^{n,+}\cup\mathbb{S}^{n,0}$.

This \textquotedblleft spatchcocking procedure\textquotedblright\ can be
repeated for every crossing index. Thus we get a controlled $(\varepsilon
,T)$-chain in $\overline{\mathbb{S}^{n,+}}$ from $s$ to $s^{\prime}$. This
shows that $_{\mathbb{S}}E_{c,j}\cap\mathbb{S}^{n,+}$ is chain controllable in
$\overline{\mathbb{S}^{n,+}}$. Since chain control sets are closed this
implies that also $_{\mathbb{S}}E_{c,j}\cap\overline{\mathbb{S}^{n,+}}$ is
chain controllable in $\overline{\mathbb{S}^{n,+}}$.
\end{proof}

The next theorem describes in a case with $\ell^{0}=1$ the chain control sets
in the closure of the northern hemisphere.

\begin{theorem}
\label{Theorem_S+}Consider an affine control system of the form (\ref{affine1}%
) on $\mathbb{R}^{n}$. Assume for the central chain control set $_{\mathbb{P}%
}E_{c}$ in $\mathbb{P}^{n}$ that the intersection with $\mathbb{P}^{n,0}$ is a
single chain control set $_{\mathbb{P}}E^{\hom}\times\{0\}$ of the homogeneous
part (\ref{bilinear0}) and that the preimage in $\mathbb{S}^{n-1}$ of
$_{\mathbb{P}}E^{\hom}$ is a single chain control set $_{\mathbb{S}}%
E_{0}^{\hom}$.

(i) Suppose that the preimage in $\mathbb{S}^{n}$ of $_{\mathbb{P}}E_{c}$ is a
single chain control set $_{\mathbb{S}}E_{c,0}$. Then the intersection
$_{\mathbb{S}}E_{c,0}\cap\overline{\mathbb{S}^{n,+}}$ is the unique chain
control set in $\overline{\mathbb{S}^{n,+}}$, which is not contained in the
equator $\mathbb{S}^{n,0}\,.$

(ii) Suppose that the preimage in $\mathbb{S}^{n}$ of $_{\mathbb{P}}E_{c}$ is
the union of two chain control sets $_{\mathbb{S}}E_{c,1}$ and $_{\mathbb{S}%
}E_{c,2}$. Then every chain control set in $\overline{\mathbb{S}^{n,+}}$,
which is not contained in $\mathbb{S}^{n,0}$, coincides with $_{\mathbb{S}%
}E_{c,1}\cap\overline{\mathbb{S}^{n,+}}$ or $_{\mathbb{S}}E_{c,2}\cap
\overline{\mathbb{S}^{n,+}}$.
\end{theorem}

\begin{proof}
For assertions (i) and (ii) the assumptions of Lemma \ref{Lemma6.1} hold.
Hence the intersections $_{\mathbb{S}}E_{c,j}\cap\overline{\mathbb{S}^{n,+}}$
are chain controllable in $\overline{\mathbb{S}^{n,+}}$. Since any chain
control set in $\overline{\mathbb{S}^{n,+}}$ having nonvoid intersection with
$\mathbb{S}^{n,+}$ is contained in some $_{\mathbb{S}}E_{c,j}$ this are the
chain control sets in $\overline{\mathbb{S}^{n,+}}$.
\end{proof}

If the assumptions of Theorem \ref{Theorem_S+} are satisfied it follows that
the images $_{\mathbb{S}}E^{\ast}$ of the strong chain control sets $E^{\ast}$
of the affine system are, in case (i) contained in $_{\mathbb{S}}E_{c,0}%
\cap\mathbb{S}^{n,+}$ and, in case (ii) in $_{\mathbb{S}}E_{c,1}\cap
\mathbb{S}^{n,+}$ or $_{\mathbb{S}}E_{c,2}\cap\mathbb{S}^{n,+}$. If the strong
chain control sets $E^{\ast}$ are unbounded (cf. Proposition
\ref{Proposition_unbounded}), the closures of the strong chain control sets
$_{\mathbb{S}}E^{\ast}$ in $_{\mathbb{S}}E_{c,0}\cap\mathbb{S}^{n,+}$ also
intersect $_{\mathbb{S}}E_{0}^{\hom}\times\{0\}$.

\begin{remark}
The assertions of Theorem \ref{Theorem_S+} also hold when the northern
hemisphere $\mathbb{S}^{n,+}$ is replaced by the southern hemisphere
$\mathbb{S}^{n,-}$.
\end{remark}

Next, we consider a situation where again the intersection of the central
chain control set $_{\mathbb{P}}E_{c}$ with $\mathbb{P}^{n,0}$ is a single
chain control set $_{\mathbb{P}}E^{\hom}$ of the homogeneous part. But now we
suppose that the preimage in $\mathbb{S}^{n-1}$ of $_{\mathbb{P}}E^{\hom}$
consists of two chain control set $_{\mathbb{S}}E_{1}^{\hom}$ and
$_{\mathbb{S}}E_{2}^{\hom}$. On the northern and southern hemispheres
$\mathbb{S}^{n,+}$ and $\mathbb{S}^{n,-}$, respectively, two chain control
sets exist. They yield two chain control sets on $\mathbb{S}^{n}$ and each of
them intersects both hemispheres.

\begin{theorem}
\label{Theorem_union}Consider an affine control system of the form
(\ref{affine1}) on $\mathbb{R}^{n}$. Assume for the central chain control set
$_{\mathbb{P}}E_{c}$ in $\mathbb{P}^{n}$ that the intersection with
$\mathbb{P}^{n,0}$ is a single chain control set $_{\mathbb{P}}E^{\hom}%
\times\{0\}$ of the homogeneous part and that the preimage in $\mathbb{S}%
^{n-1}$ of $_{\mathbb{P}}E^{\hom}$ is the union of two chain control sets
$_{\mathbb{S}}E_{1}^{\hom}$ and $_{\mathbb{S}}E_{2}^{\hom}$. Furthermore, let
$_{\mathbb{S}}E_{1}^{+}$ and $\,_{\mathbb{S}}E_{2}^{+}$ be two chain control
sets in the northern hemisphere $\mathbb{S}^{n,+}$ such that%
\[
\overline{_{\mathbb{S}}E_{1}^{+}}\cap\left(  _{\mathbb{S}}E_{1}^{\hom}%
\times\{0\}\right)  \not =\varnothing\text{ and }\overline{_{\mathbb{S}}%
E_{2}^{+}}\cap\left(  _{\mathbb{S}}E_{2}^{\hom}\times\{0\}\right)
\not =\varnothing.
\]
Then, with the chain control sets $_{\mathbb{S}}E_{1}^{-}:=-\,_{\mathbb{S}%
}E_{1}^{+}$ and $_{\mathbb{S}}E_{2}^{-}:=-\,_{\mathbb{S}}E_{2}^{+}$ in the
southern hemisphere $\mathbb{S}^{n,-}$, both sets $_{\mathbb{S}}E_{1}^{+}%
\cup\,_{\mathbb{S}}E_{2}^{-}$ and $_{\mathbb{S}}E_{2}^{+}\cup\,_{\mathbb{S}%
}E_{1}^{-}$ are contained in chain control sets in $\mathbb{S}^{n}$.
\end{theorem}

\begin{proof}
Since $_{\mathbb{S}}E_{2}^{-}=-\,_{\mathbb{S}}E_{2}^{+}$ and $_{\mathbb{S}%
}E_{1}^{\hom}=-\,_{\mathbb{S}}E_{2}^{\hom}$ it follows that%
\[
\overline{_{\mathbb{S}}E_{2}^{-}}\cap\left(  _{\mathbb{S}}E_{1}^{\hom}%
\times\{0\}\right)  =-\overline{\,_{\mathbb{S}}E_{2}^{+}}\cap\left(
-\,_{\mathbb{S}}E_{2}^{\hom}\times\{0\}\right)  =-\left(  \overline
{_{\mathbb{S}}E_{2}^{+}}\cap\left(  _{\mathbb{S}}E_{2}^{\hom}\times
\{0\}\right)  \right)  \not =\varnothing.
\]
This implies that the union
\[
\,_{\mathbb{S}}E_{1}^{+}\cup\,_{\mathbb{S}}E_{2}^{-}\cup\left(  _{\mathbb{S}%
}E_{1}^{\hom}\times\{0\}\right)
\]
is a chain controllable set, hence contained in a chain control set in
$\mathbb{S}^{n}$. The assertion for $_{\mathbb{S}}E_{2}^{+}\cup\,_{\mathbb{S}%
}E_{1}^{-}$ follows analogously.
\end{proof}

The following example (cf. Colonius, Santana, and Setti \cite[Example
5.17]{ColSS22}, \cite[Example 6.10]{ColSS23}) illustrates Theorem
\ref{Theorem_union}. It shows the delicate relations between chain control
sets and strong chain control sets in the various spaces $\mathbb{S}^{n}$ and
$\mathbb{P}^{n}$ as well as in the hemispheres $\mathbb{S}^{n,+}$ and
$\mathbb{S}^{n,-}$. The affine systems on $\mathbb{R}^{2}$ has two strong
chain control sets $E_{1}^{\ast}$ and $E_{2}^{\ast}$, hence also the induced
systems on $\mathbb{S}^{n,+}$ and $\mathbb{P}^{n,1}$ have two strong chain
control sets, while the system on $\mathbb{P}^{2}$ has a single (strong) chain
control set. See Figure \ref{fig2}.

\begin{example}
\label{Example2}Consider the affine control system%
\[
\left(
\begin{array}
[c]{c}%
\dot{x}\\
\dot{y}%
\end{array}
\right)  =\left(
\begin{array}
[c]{cc}%
0 & 1\\
-1-u(t) & -3
\end{array}
\right)  \left(
\begin{array}
[c]{c}%
x\\
y
\end{array}
\right)  +u(t)\left(
\begin{array}
[c]{c}%
0\\
1
\end{array}
\right)  +\left(
\begin{array}
[c]{c}%
0\\
d
\end{array}
\right)  ,\quad u(t)\in\Omega:=[-\rho,\rho],
\]
where we suppose $\rho\in\left(  1,\frac{5}{4}\right)  $ and $d<1$. The
equilibria are%
\begin{align*}
\mathcal{C}_{1}  &  =\left\{  \left.  \left(
\begin{array}
[c]{c}%
x\\
0
\end{array}
\right)  \right\vert x\in\left[  \frac{d-\rho}{1-\rho},\infty\right)
\right\}  =\left\{  \left.  \left(
\begin{array}
[c]{c}%
\frac{d+u}{1+u}\\
0
\end{array}
\right)  \right\vert u\in\lbrack-\rho,-1)\right\}  ,\\
\mathcal{C}_{2}  &  =\left\{  \left.  \left(
\begin{array}
[c]{c}%
x\\
0
\end{array}
\right)  \right\vert x\in\left(  -\infty,\frac{d+\rho}{1+\rho}\right]
\right\}  =\left\{  \left.  \left(
\begin{array}
[c]{c}%
\frac{d+u}{1+u}\\
0
\end{array}
\right)  \right\vert u\in(-1,\rho]\right\}  ,
\end{align*}
respectively. There is no equilibrium for $u=-1$. Note that $\frac{d-\rho
}{1-\rho}>\frac{d+\rho}{1+\rho}>0$, hence $\mathcal{C}_{1}\cap\mathcal{C}%
_{2}=\varnothing$. The eigenvalues of the linear part are $\lambda
_{1,2}(u)=-\frac{3}{2}\pm\sqrt{\frac{5}{4}-u},u\in\lbrack-\rho,\rho]$. For
$u=-1$ the eigenvalues are $\lambda_{1}(-1)=0,\lambda_{2}(-1)=-3$ and the
eigenspace of $\lambda_{1}(-1)$ is $\mathbb{R}\times\{0\}$. For $u^{k}\nearrow
u^{0}=-1$, the equilibria in $\mathcal{C}_{1}$ satisfy $(x_{u^{k}%
},0)\rightarrow(\infty,0)$ and for $u^{k}\searrow u^{0}=-1$, the equilibria in
$\mathcal{C}_{2}$ satisfy $(x_{u^{k}},0)\rightarrow(-\infty,0)$ for
$k\rightarrow\infty$. The equilibria in $\mathcal{C}_{1}$ are hyperbolic since
for $u\in\lbrack-\rho,-1)\subset(-\frac{5}{4},-1)$ one obtains $\lambda
_{1}(u)>0>\lambda_{2}(u)$. The equilibria in $\mathcal{C}_{2}$ are stable
nodes since for $u\in(-1,\rho]$ one obtains $0>\lambda_{1}(u)>\lambda_{2}(u)$.
For constant $u$ the phase portrait with $y:=x-\frac{u+d}{1+u}$ is the shifted
phase portrait of%
\[
\ddot{y}+3\dot{y}+(1+u)y=0.
\]
There are control sets $D_{1}$ and $D_{2}$ containing the equilibria in
$\mathcal{C}_{1}$ and $\mathcal{C}_{2}$, respectively, in the interior.

While the asymptotic stability of the equilibria in $\mathcal{C}_{2}$ implies
that one can steer the system from $D_{1}$ to $D_{2}$, the converse does not
hold which follows by inspection of the phase portraits for the controls in
$[-\rho,-1]$ and $\left[  -1,\rho\right]  $, hence $D_{1}\not =D_{2}$. One
also sees that the unbounded chain control sets $E_{1}$ and $E_{2}$ with
$D_{1}\subset E_{1}$ and $D_{2}\subset E_{2}$ are disjoint. Note that
according to Gayer \cite[Theorem 19 and Remark 20]{Gayer04} the system above
satisfies the so-called inner-pair condition, which in the compact
case guarantees that, generically, the closures of control sets are chain
control sets. The continua of equilibria $\mathcal{C}_{1}$ and $\mathcal{C}%
_{2}$ in $\mathrm{int}D_{1}$ and $\mathrm{int}D_{2}$ are contained in strong
chain control sets with $E_{1}^{\ast}\cap D_{1}\not =\varnothing$ and
$E_{2}^{\ast}\cap D_{2}\not =\varnothing$, respectively. It follows that
\[
D_{1}\subset E_{1}^{\ast}\subset E_{1}\text{ and }D_{2}\subset E_{2}^{\ast
}\subset E_{2}.
\]
Hence the images $h(E_{1}^{\ast})$ and $h(E_{2}^{\ast})$ are in different
strong chain control sets in $\mathbb{S}^{2,+}$. With $h_{-}:\mathbb{R}%
^{2}\rightarrow\mathbb{S}^{2,-}$ defined in Remark \ref{Remark_southern} one
similarly obtains that $h_{-}(E_{1}^{\ast})$ and $h_{-}(E_{2}^{\ast})$ are in
different strong chain control sets in $\mathbb{S}^{2,-}$.

On the other hand, one has $\lim_{x\rightarrow\pm\infty}h(x,0)=(\pm1,0,0)$,
and
\[
\mathbb{S}^{2,0}\ni(1,0,0)\in\overline{h(E_{1}^{\ast})}\cap\overline
{h_{-}(E_{2}^{\ast})}\text{ and }\mathbb{S}^{2,0}\ni(-1,0,0)\in\overline
{h(E_{2}^{\ast})}\cap\overline{h_{-}(E_{1}^{\ast})}.
\]
Thus the subsets of $\mathbb{S}^{n}$ given by
\[
\overline{h(E_{1}^{\ast})}\cup\overline{h_{-}(E_{2}^{\ast})}\text{ and
}\overline{h(E_{2}^{\ast})}\cup\overline{h_{-}(E_{1}^{\ast})}%
\]
are strongly chain controllable. Each of these two subsets intersects both
hemispheres $\mathbb{S}^{2,+}$ and $\mathbb{S}^{2,-}$. In projective space
$\mathbb{P}^{2,0}$, the points $\mathbb{P}(1,0,0)$ and $\mathbb{P}(-1,0,0)$
coincide. It follows for the images of the strong chain control sets
$E_{1}^{\ast}$ and $E_{2}^{\ast}$ that their closures in $\mathbb{P}%
^{2}=\mathbb{P}^{2,1}\cup\mathbb{P}^{2,0}$ intersect, hence they are contained
in a single strong chain control set of the compact space $\mathbb{P}^{2}$,
i.e., in the central chain control set $_{\mathbb{P}}E_{c}$. 
\end{example}


\begin{figure}[t]
\centering
\includegraphics[scale=0.6]{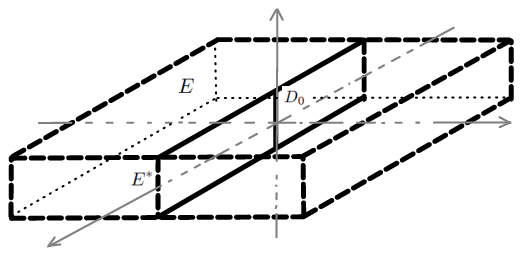}
\caption{Control set $D_{0}=\{(0,0)\}\times(-1,1)$,  strong chain control set $E^{\ast}=\mathbb{R}\times\{0\}\times\lbrack-1,1]$ and chain control set $E=\mathbb{R}^{2}\times\lbrack-1,1]$
	in Example \ref{Example_linear}.}\label{fig1}
\end{figure}

\begin{figure}[t]
\centering
\includegraphics[scale=0.6]{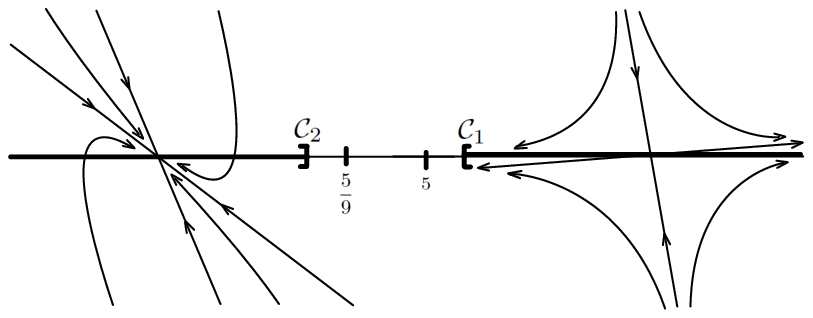}
\caption{Phase portraits in Example \ref{Example2} for constant $u$-values yielding equilibria in $\mathcal{C}_{1}$ and $\mathcal{C}_{2}$.}\label{fig2}
\end{figure}





\end{document}